\documentclass[aihp]{imsart}

\usepackage{natbib}
\usepackage[colorlinks,citecolor=blue,urlcolor=blue]{hyperref}
\usepackage[utf8]{inputenc}
\usepackage[T1]{fontenc}
\usepackage[english]{babel}
\usepackage{charter}
\usepackage{enumerate}
\usepackage{indentfirst}
\usepackage{xcolor}
\usepackage{geometry}
\usepackage{mathtools}

\usepackage{amsmath}
\usepackage{amsthm}	
\usepackage{amsfonts}
\usepackage{mathrsfs}	
\usepackage{amssymb}
\usepackage{dsfont}
\usepackage{bbm}

\startlocaldefs

\theoremstyle{plain}
\newtheorem{theorem}{Theorem}[section]
\newtheorem{lemma}[theorem]{Lemma}
\newtheorem{proposition}[theorem]{Proposition}
\newtheorem{corollary}[theorem]{Corollary}
\theoremstyle{definition}
\newtheorem{definition}[theorem]{Definition}

\newtheorem{example}[theorem]{Example}
\newtheorem{assumption}[theorem]{Assumption}

\theoremstyle{definition}
\newtheorem{remark}[theorem]{Remark}

\numberwithin{equation}{section}
\numberwithin{figure}{section}
 \usepackage[nodayofweek]{datetime}

\newcommand{\bE}{\mathbb{E}}

\newcommand{\bN}{\mathbb{N}}
\newcommand{\bP}{\mathbb{P}}

\newcommand{\bR}{\mathbb{R}}

\def \E {\mathbb{E}}

\newcommand{\cB}{\mathcal{B}}

\newcommand{\cF}{\mathcal{F}}

\newcommand{\cK}{\mathcal{K}}

\newcommand{\caL}{\mathscr{L}}

\newcommand{\cP}{\mathcal{P}}

\renewcommand{\P}{\mathbb{P}}

\def \a {\alpha}

\newcommand{\1}{\mathbbm{1}}

\endlocaldefs

\usepackage[disable,colorinlistoftodos,bordercolor=orange,backgroundcolor=orange!20,linecolor=orange,textsize=tiny]{todonotes}

\newcommand{\firstrefereee}[1]{\todo[inline,backgroundcolor=blue!30]{\textbf{To Referee 1: }#1}}
\newcommand{\secondrefereee}[1]{\todo[inline,backgroundcolor=green!30]{\textbf{To Referee 2: }#1}}
\newcommand{\bothreferees}[1]{\todo[inline]{\textbf{To Both Referees: }#1}}

\begin{document}

\begin{frontmatter}

\title{McKean--Vlasov SDEs under Measure Dependent Lyapunov Conditions}
\runtitle{McKean--Vlasov SDEs under Measure Dependent Lyapunov Conditions}

\begin{aug}
  \author{William R.P. Hammersley\corref{1} \thanksref{t1}$^{,1}$\ead[label=e1]{w.r.p.hammersley@sms.ed.ac.uk}},
  \author{David \v{S}i\v{s}ka\corref{2}$^{2}$
  \ead[label=e2]{d.siska@ed.ac.uk}}
  \and
  \author{{\L}ukasz Szpruch \corref{3}$^{3}$
  \ead[label=e3]{l.szpruch@ed.ac.uk}}
  \thankstext{t1}{Supported by The Maxwell Institute Graduate School in Analysis and its Applications, a Centre for Doctoral Training funded by the UK Engineering and Physical Sciences Research Council (grant EP/L016508/01), the Scottish Funding Council, Heriot-Watt University and the University of Edinburgh.}
  
 
 
  \affiliation{School of Mathematics, University of Edinburgh$^{1,2,3}$ and the Alan Turing Institute, London$^3$}
  \address{School of Mathematics, University of Edinburgh,\\ 
  	    James Clerk Maxwell Building, Peter Guthrie Tait Road,\\ Edinburgh EH9 3FD\\
        \printead{e1,e2,e3}}

\runauthor{W.R.P. Hammersley, D. \v{S}i\v{s}ka and {\L}. Szpruch}

\end{aug}

\begin{abstract}
We prove the existence of weak solutions to McKean--Vlasov SDEs defined on a domain $D \subseteq \mathbb{R}^d$ with continuous and unbounded 
coefficients and degenerate diffusion coefficient. 
Using differential calculus for the flow of probability measures due to Lions, we introduce a novel integrated condition for Lyapunov functions in an infinite dimensional space $D\times \cP(D)$, where $\cP(D)$ is a space of probability measures on $D$. Consequently we show existence of solutions to the McKean--Vlasov SDEs on $[0,\infty)$. 
This leads to a probabilistic proof of the existence of a stationary solution to 
the nonlinear Fokker--Planck--Kolmogorov equation under very general conditions. 
Finally, we prove uniqueness under an integrated condition based on a Lyapunov function. 
This extends the standard monotone-type condition for uniqueness.   
\end{abstract}




\end{frontmatter}


\section{Introduction}

We will consider either the time interval $I=[0,T]$ for some fixed
$T>0$ or $I=[0,\infty)$.
Let $(\Omega, \mathcal{F}, \mathbb{P})$ be a probability space and
$(\mathcal{F}_t)_{t\in I}$ 
a right continuous filtration such that $\mathcal{F}_0$ contains all sets of $\mathcal{F}$ that
have probability zero.
Let $w=(w_t)_{t\in I}$ be an $\mathbb{R}^{d'}$-valued an $(\mathcal F_t)_{t\in I}$-Wiener process.
We consider the McKean--Vlasov stochastic differential equation (SDE) on an open domain $D\subseteq \mathbb R^d$, 
\begin{equation}
\label{eq mkvsde}	
x_t = x_0 + \int_0^t b(s,x_s,\mathscr{L}(x_s))\,ds 
+ \int_0^t \sigma(s,x_s,\mathscr{L}(x_s))\,dw_s\,,\,\,\, t\in I\,.
\end{equation}
Here we use the notation $\mathscr{L}(x)$ to denote the
law of the random variable $x$.
The law of such an SDE satisfies a nonlinear Fokker--Planck--Kolmogorov equation (see also~\cite{bogachev2016distances} and more generally~\cite{BogachevKrylovRocknerShaposhnikov}): writing $\mu_t := \mathscr L(x_t)$ and $a:= \frac12\sigma \sigma^*$ we have, for $t\in I$,
\begin{equation}
\label{eq fwd kolmogorov}
\langle \mu_t, \varphi \rangle 
= \langle \mu_0, \varphi \rangle 
+ \int_0^t \left \langle \mu_s, b(s,\cdot, \mu_s)\partial_x \varphi 
+ \text{tr}\left(a(s,\cdot,\mu_s) \partial_x^2 \varphi\right) 	\right \rangle \, ds\,\,\,\, \forall \varphi \in C^2_0( D )\,.
\end{equation}

The aim of this article is to study the existence and uniqueness of solutions to 
the equation~\eqref{eq mkvsde}. 
We will show that a weak solution to~\eqref{eq mkvsde} exists for unbounded and
continuous coefficients, provided that we can find an 
appropriate measure-dependent Lyapunov function which ensures integrability of the equation. 
This generalises the results of~\cite{funaki1984certain} and~\cite{gyongy1996existence}.

The work on SDEs with coefficients that depend on the law of the solution was initiated 
by McKean~\cite{McKean66}, who was inspired by Kac's programme in Kinetic Theory~\cite{Kac56}. 
An excellent and thorough account of the general theory of McKean--SDEs and their particle approximations can be found in~\cite{sznitman1991topics}. Sznitman has shown that if the coefficients of \eqref{eq mkvsde} are globally Lipschitz continuous, a fixed point argument on Wasserstein space can be carried out, and consequently a solution to \eqref{eq mkvsde} is obtained as the limit of classical SDEs. 
To extend this result, Funaki~\cite{funaki1984certain} formulated 
a non-linear martingale problem for McKean--Vlasov SDEs that allowed him to establish existence of a solution to~\eqref{eq mkvsde} by studying a limiting law of  Euler discretisation. His proof of existence holds for continuous coefficients satisfying a Lyapunov type condition in the state variable $x\in \bR^d$ with polynomial Lyapunov functions.  Whilst we also assume continuity of the coefficients, we allow for a much more general Lyapunov condition that depends on a measure. Furthermore, Funaki is using Lyapunov functions to establish integrability of the Euler scheme which is problematic if one wants to depart from polynomial functions, see \cite{SzpruchZhang18}. G\"artner \cite{Gartner}, uses an integrated Lyapunov condition with a Lyapunov function not dependent on measure, to study the weak well-posedness of McKean--Vlasov SDEs.
\firstrefereee{established changed to establish}

An alternative approach to establishing existence of solutions to McKean--Vlasov
equations is to approximate the equation with a particle system 
(a system of classical SDEs that interact with each other through empirical measure) and show that the limiting law solves the martingale problem.
In this approach, one works with laws of empirical laws i.e. on the space of probability measures on the space of probability measures on $D$ - denoted
$\cP(\cP(D))$ - and proves their convergence to a (weak) solution of~\eqref{eq mkvsde} by studying the corresponding non-linear martingale problem. 
We refer to~\cite{meleard1996asymptotic} for a general overview and to~\cite{bossy2011conditional,FournierJourdain17} and references within for recent results exploring this method. A general approach to establish the existence of martingale solutions has also been presented in  \cite{MR3609379}.
We also refer the reader to interesting new developments on existence and uniqueness of solutions for McKean-Vlasov equations with non-smooth coefficients to \cite{mishura2016existence,de2015strong}.
Here, inspired by~\cite{mishura2016existence}, 
we tackle the problem using the Skorokhod representation theorem and convergence
lemma~\cite{skorokhod1965}.
\bothreferees{Rewrote sentence beginning with 'We refer to' in the above paragraph}
\firstrefereee{Full stop added in l-3 of paragraph}
\secondrefereee{description of the space $\cP(\cP(D))$ added.}

For classical SDEs (equations with no dependence on the law), the lack of sufficient regularity of the coefficients, say Lipschitz continuity, proves to be the main challenge in establishing existence and uniqueness of solutions. Lack of boundedness of the coefficients, typically, does not lead to significant difficulty, provided these are at least locally bounded. In that case one can work with local solutions and the only concern is the possible explosion. The conditions that ensure that the solution does not explode can be formulated by using Lyapunov function techniques as has been pioneered in \cite{Khasminskii80}. The key observation is that if one considers two SDEs with coefficients  that agree on some bounded domain then the solutions, if unique, also agree until first time the solution leaves the domain, see, for example~\cite[Ch. 10]{StroockVaradhan2006}.

This classical localisation procedure does not carry over, at least directly, from the setting of classical SDEs to McKean--Vlasov SDEs. 
Indeed, if we stop a classical SDE then until the stopping time the stopped
process satisfies the same equation. 
If we take~\eqref{eq mkvsde} and consider the stopped process 
$y_t := x_{t\wedge \tau}$, with some stopping time $\tau$, then the equation this satisfies is 
\firstrefereee{'its satisfies is' changed to 'this process satisfies is'}
\[
y_t = y_0 + \int_0^{t\wedge \tau} b(s,y_s,\mathscr{L}(x_s))\,ds 
+ \int_0^{t\wedge \tau} \sigma(s,y_s,\mathscr{L}(x_s))\,dw_s\,,\,\,\, t\in I\,.
\]
Clearly, even for $t\leq \tau$ this is not the same equation 
since $\mathscr L(x_s) \neq \mathscr L(y_s)$. Furthermore, this is not a McKean--Vlasov SDE. This could be problematic if one would like to obtain a solution to McKean--Vlasov SDEs through a limiting procedure of stopped processes. Furthermore, let $D_k\subseteq D_{k+1}$ be a sequence of nested domains,  and consider functions $\bar b$ and $\bar \sigma$ such that $\bar b = b $ and $\bar \sigma = \sigma$ on $D_k$. The equation
\begin{equation*}	
\bar x_t =\bar x_0 + \int_0^t \bar b(s,\bar x_s,\mathscr{L}(\bar x_s))\,ds 
+ \int_0^t \bar \sigma(s,\bar x_s,\mathscr{L}(\bar x_s))\,dw_s\,,\,\,\, t\in I\,,
\end{equation*}
is a McKean--Vlasov SDE, but in general $x_t \neq \bar x_t$ even for $t\leq \bar \tau^k$, where $\bar \tau^k = \inf\{t \geq 0: \bar x_t \notin D_k \}$. This implies that if one considers a sequence of SDEs with coefficients that agree on these subdomains, one no longer has monotonicity for the corresponding stopping times.  
We show that despite these difficulties it still possible to establish the existence of weak solutions to the McKean--Vlasov SDEs \eqref{eq mkvsde} using the idea of localisation, but extra care is needed. 
\secondrefereee{Caveat 'in general added to $x_t\neq \bar{x}_t$}

\subsection{Main Contributions}

Our first main contribution is the generalisation of Lyapunov function techniques to the setting of McKean--Vlasov SDEs. 
The coefficients of the equation~\eqref{eq mkvsde} depend on  $(x,\mu)\in D \times \cP(D)$ for $D\subseteq \mathbb{R}^d$.  
Hence the class of Lyapunov functions considered in this paper also depend on $(x,\mu)\in D \times \cP(D)$. 
See~\eqref{a-nonint}.
Furthermore, it is natural to formulate the {\em integrated Lyapunov} condition,
in which the key stability assumption is required to hold only on $\mathcal P(D) $, see~\eqref{a-int} and Section~\ref{sec:motivating examples} for motivating examples.
Note that it is not immediately clear how one can obtain tightness estimates
for the particle approximation under the integrated conditions we propose.
To work with Lyapunov functions on $\mathcal P(D)$, we take advantage of the recently developed analysis on Wasserstein spaces, and in particular derivatives with respect to a measure
as introduced by Lions in his lectures at College de France,
see~\cite{cardaliaguet2012} and~\cite[Ch. 5]{carmona2017probabilistic}. This analysis is presented in the appendix to give the measure derivative in a domain. 
\firstrefereee{Comma removed in second sentence}
Our second main contribution is the probabilistic proof of the existence of a stationary solution to the nonlinear Fokker--Planck--Kolmogorov equation \eqref{eq fwd kolmogorov}. 
Furthermore, the calculus on Wasserstein spaces allows one to study a Fokker--Planck--Kolmogorov-type equation on $\mathcal P_2(D)$. 
Indeed, writing $\mu_t:=\mathscr L(x_t)$ we have, for $\phi\in  \mathcal C^{(1,1)}(\mathcal P_2(D))$, see Definition~\ref{def a5}, and $t\in I$ that
\begin{equation} \label{eq fwd measure}
\begin{split}
\phi(\mu_t)=\phi(\mu_0) \, +  
\int_0^t \big\langle \mu_s,  b(s,\cdot,\mu_s) \partial_\mu \phi(\mu_s) 
+ \text{tr}\left[a(s,\cdot,\mu_s) \partial_y \partial_\mu \phi(\mu_s)\right]\big\rangle\,ds.
\end{split}
\end{equation}
Following the remark by Lions from his lectures at College de France, the equation~\eqref{eq fwd measure} can be interpreted as non-local transport equation on the space of measures. 
The reader may consult~\cite[Ch. 5 Sec. 7.4]{carmona2017probabilistic} for further details. 
Another interesting angle is that whilst~\eqref{eq fwd kolmogorov} gives
an equation for linear functionals of the measure, 
equation~\eqref{eq fwd measure} is an equation for nonlinear functionals of the measure.
The existence results obtained in this paper imply existence of a stationary solution to \eqref{eq fwd measure} in the case where $b$ and $\sigma$ do not
depend on time.

Finally, we formulate uniqueness results under the Lyapunov type condition and the {\em integrated Lyapunov type condition} that is required to hold only on $\mathcal P(D) $.
This extends the standard monotone type conditions studied in literature e.g~\cite{MR2860672,wangDDSDELandau2018,MR3226169,DosReis2017LDP}. Interestingly, in some special cases we are able to obtain uniqueness only under local monotone conditions. Again, we do not require a non-degeneracy condition on the diffusion coefficient. We support our results with the example inspired by Scheutzow~\cite{Scheutzow87} who has shown that, in general, uniqueness of solution to McKean--Vlasov SDEs does not hold if the coefficients are only locally Lipschitz. Again, we would like to highlight that since the standard localisation techniques used for classical SDEs seem not to work in our setting, we cannot simply obtain global uniqueness results from local uniqueness and suitable estimates on the stopping times.   

\subsection{Motivating Examples}
\label{sec:motivating examples}

Let us now present some examples to motivate the choice of the Lyapunov condition.
Consider first the McKean--Vlasov stochastic differential equation
\begin{equation} 
\label{eq example1}
d x_t  = -x_t \bigg[\int_\bR y^4 \mathscr L(x_t)(dy)\bigg]\, dt + \frac{1}{\sqrt{2}}x_t\,  d w_t\,,\,\,\,
x_0 \in L^4(\mathcal F_0, \mathbb R^+)\,.
\end{equation}
The diffusion generator for~\eqref{eq example1} is
\begin{equation}
\label{eq ex diff gen}
L(x,\mu) v(x) := \frac{1}{4}x^2  v''(x) - x \bigg[ \int_{\mathbb R} y^4 \, \mu(dy) \bigg]  v'(x)\,.	
\end{equation}
It is not clear whether one can find a Lyapunov function such that the classical Lyapunov condition holds i.e. $L(x,\mu)v(x)\leq  m_1 v(x)+  m_2$, for $m_1<0$ and $m_2\in \bR$. However, with the Lyapunov function given by $v(x) = x^4$ we can establish that
\bothreferees{$m_1v(x)m_2$ changed to $m_1v(x)+m_2$} 
\begin{equation}
\label{eq lyapunov example}
\int_{\bR}L(x,\mu) v(x)\mu(dx) 
\leq -\int_\bR v(x) \mu(dx) + 1.
\end{equation}
\bothreferees{from the observation of the second reviewer we changed $+4$ to $+1$. }
See Example~\ref{example integrated lyapunov} for details.
We will see that this is sufficient to establish integrability of~\eqref{eq example1} on $I=[0,\infty)$. 
See Theorem~\ref{thm:weakexistence} and the condition~\eqref{eq b1}.

Another way to proceed, is to directly work with $v(\mu):=\int_\bR x^4 \, \mu(dx)$
as Lyapunov function on the measure space $\cP_4(\bR)$. 
This requires the use of derivatives with respect to a measure
as introduced by Lions in his lectures at College de France,
see~\cite{cardaliaguet2012} or Appendix~\ref{sec measure derivatives and ito}. We note that derivatives with respect to a measure are defined in $\cP_2(\bR)$, and therefore one cannot apply It\^{o} formula for arbitrary measures in $\cP(D)$. However, in this paper we will only apply the It\^{o} formula for measures supported on compact subsets of $\mathbb R^d$ and hence, measures in $\cP_2(\bR)$.
Then 
\[
\partial_\mu v(\mu)(y) = 4 y^3, \quad \partial_y\partial_\mu v(\mu)(y) = 12 y^2, \,\, y \in \bR\,. 
\]
The generator corresponding to the appropriate It\^o formula, see e.g. Proposition~\ref{propn ito for meas only},
is
\begin{align*}
 L^{\mu} v(\mu)
  := & \int_\bR \left( -x \int_\bR y^4\, \mu(dy) \partial_\mu v(\mu)(x)  +\frac{1}{4}x^2  \partial_y \partial_\mu v(\mu)(x)   \right) \mu(d x) \\
= &  \int_\bR \left( - 4 x^4   \int_\bR y^4 \mu(dy) + 3x^{4}   \right) \mu(d x)\,. 
\end{align*}
\bothreferees{$x^4$ changed to $x^2$}
We note that this yields the same expression as found when $v(x) = x^4$ 
in~\eqref{eq ex diff gen} after we integrate over $\mu$ (and so~\eqref{eq lyapunov example} again holds). 
In this case using the It\^o formula for measure derivatives brings no advantages. 
However, the advantage of working with a Lyapunov function on the measure space appears when the dependence on the measure in the Lyapunov function is not linear.

Consider the following McKean--Vlasov stochastic differential equation
\begin{equation}
\label{eq example2}
dx_t = - \left( \int_\bR (x_t - \a y)\caL{(x_t)}(dy) \right)^3\, dt + \left( \int_\bR (x_t - \a y)\caL{(x_t)}(dy) \right)^2 \sigma \, dw_t\,, 	
\end{equation}
for $t\in I$, $\alpha$ and $\sigma$ constants and with $x_0 \in L^4(\cF_0,\mathbb R)$.
Assume that $m:= -(6\sigma^2 - 4 + 4\alpha) >0$.
Since the drift and diffusion are non-linear functions of the law and state of the process, it is natural to seek a Lyapunov function 
$v \in \mathcal C^{2,(1,1)}(\bR \times \mathcal{P}(\bR))$. 
See Definition~\ref{def c122}.
The generator corresponding to the appropriate It\^o formula, see e.g. Proposition~\ref{propn ito formula full},
is then given by~\eqref{eq Lmu} and we will show that 
for the Lyapunov function
\[
v(x,\mu) = \left( \int_{\bR} (x - \a y)\mu(dy)  \right)^4\,,
\]  
we have 
\[
\int_{\bR}(L^\mu v)(x,\mu)\,\mu(dx) \leq m  - m  \int_{\bR}v(x,\mu)\,\mu(dx)\,. 
\]
See Example~\ref{example 2} for details.
Thus the condition~\eqref{eq b1} holds.
This is sufficient to establish existence of solutions to~\eqref{eq example1} on $I=[0,\infty)$ as Theorem~\ref{thm:weakexistence} will tell us. 

Regarding our continuity assumptions for existence of solutions to~\eqref{eq mkvsde} we note that we only require a type of joint continuity of the coefficients in $(x,\mu) \in \mathbb R^d \times \mathcal P(\mathbb R^d)$
and that this allows us to consider coefficients where the dependence
on the measure does not arise via an integral with respect to said measure.
This could be for example 
\[
S_\alpha(\mu):=\frac{1}{\alpha}\int_0^\alpha \inf\{x\in \bR\,:\, \mu((-\infty, x]) \geq s]\}\,ds\,,
\]
for $\alpha > 0$ fixed.
This quantity is known as the ``expected shortfall'' and is a type of risk measure. See Example~\ref{example 3} for details. 
These motivating examples also satisfy the Lyapunov type estimates, appearing in Section~\ref{sec uniq}, 
ensuring uniqueness of solutions.

\section{Existence Results}

For a domain $D\subseteq \mathbb R^d$, we will use the notation $\mathcal{P}(D)$ for the space
of probability measures over $(  D,\mathcal{B}(D))$. 
We will consider this as a topological space with the topology induced
by the weak convergence of probability measures.
We will write $\mu_n \Rightarrow \mu$ if $(\mu_n)_n$ converges to $\mu$ 
in the sense of weak convergence of probability measures. 
For $p\geq 1$ we use $\mathcal{P}_p(D)$ to denote the set of probability 
measures on $D$ with finite $p^{th}$ moment (i.e. $\int_D |x|^p \mu(dx) < \infty$ for $\mu \in \mathcal P_p(D)$).
We will consider this as a metric space with the metric given by the $p^{th}$ Wasserstein distance, see~\eqref{eq p-wasserstein}.
Denote by $C_b(D)$ and $C_0(D)$ the subspaces of continuous functions that are bounded and compactly supported, respectively.
 \firstrefereee{Changed 'set of probability measures that are p integrable' to 'set of probability measures on $D$ with finite $p^{th}$ moment'. Also changed 'Wasserstein distance with exponent $p$' to '$p^{th}$ Wasserstein distance'. }

We use $\sigma^*$ to denote the transpose of a matrix $\sigma$ 
and for a square matrix $a$ we use $\text{tr}(a)$ to denote its trace.
We use 
$\partial_x v$ to denote the (column) vector of first order partial derivatives of $v$ with respect to the components of $x$ (i.e. the gradient of $v$ with respect to $x$) 
and $\partial_x^2 v$ to denote the square matrix of all the mixed second order partial derivatives with respect to the components of $x$ (i.e. the Hessian matrix of $v$ with respect to $x$). 
If $a,b \in \mathbb R^d$ then $ab$ denotes their dot product.

Recall that we are using the concept of derivatives with respect to a measure
as introduced by Lions in his lectures at Coll${\grave{\textnormal e}}$ge  de France,
see~\cite{cardaliaguet2012}. 
For convenience, the construction and main definitions are 
in Appendix~\ref{sec measure derivatives and ito}.
In particular, see Definition~\ref{def c122} to clarify what is meant by
the space
$\mathcal C^{1,2,(1,1)}(I \times D \times \mathcal{P}(D))$.
In short, saying that a function $v$ is in such space means that all the derivatives appearing in~\eqref{eq Lmu}
exist and are appropriately jointly continuous so that we may apply 
the It\^o formula for a function of a process and a flow of measures, 
see Proposition~\ref{propn ito formula full}.
The use of such an It\^o formula naturally leads to the following form of a diffusion generator. First, we note that throughout this paper we assume that for the domain $D\subseteq{\mathbb R^d}$ there is a nested sequence of bounded sub-domains. By this we mean a sequence of bounded open subsets of $\mathbb R^d$, $(D_k)_k$ such that $\bigcup_k D_k = D$ and $\overline{D}_k\subset D_{k+1}$ for all $k$, i.e. $d( D_k,\partial D_{k+1}):=\inf_{x\in D_k,y\in \partial D_{k+1}}|x-y|>0$ for all $k\in \bN$.
\bothreferees{We have corrected the assumptions on $D_k$ to ensure that the limiting argument for the stopping times introduced later holds.}
For $(t,x)\in I \times D$, $\mu \in \mathcal P( D_k)$ for some $k\in \mathbb N$ and
for some 
$v \in \mathcal C^{1,2,(1,1)}(I \times D \times \mathcal{P}_2(D))$  
we define the diffusion generator $L^{\mu} = L^\mu(t,x,\mu)$ as
\secondrefereee{$L^\mu$ is emphasized to be the diffusion generator here in response to the 16th point raised in your report}
\begin{equation}
\label{eq Lmu}
\begin{split}
(L^\mu v)(t,x,\mu) & :=  \bigg (\partial_t v 
+ \frac{1}{2}\text{tr}\big (\sigma \sigma^*\partial_x^2  v\big ) 
+ b \partial_x v\bigg )(t, x,\mu) \\
& + \int_{\bR^d}\left( b (t, y, \mu) (\partial_\mu v) (t,x,\mu)(y)  
+ \frac{1}{2}\text{tr}\big ((\sigma\sigma^*)(t,y,\mu)(\partial_y \partial_\mu v)(t,x,\mu)(y) \big )   \right) \mu(dy).	
\end{split}
\end{equation}
We note that in the case $v \in C^{1,2}(I \times D )$, i.e when $v$ does not depend on the measure, the above generator reduces to 
\begin{equation*}
(L^\mu v)(t,x) = (Lv)(t,x) :=  \bigg (\partial_t v 
+ \frac{1}{2}\text{tr}\big (\sigma \sigma^*\partial_x^2  v\big ) 
+ b \partial_x v\bigg )(t,x) \,.
\end{equation*}

\subsection{Assumptions and Main Result}
We assume that $b:I\times D \times \mathcal{P}(D) \to \mathbb{R}^d$ and 
$\sigma:I\times D \times \mathcal{P}(D)\to \mathbb{R}^d\times \mathbb{R}^{d'}$
are measurable (later we will add joint continuity and local boundedness assumptions). We require the existence of a Lyapunov function satisfying one of the following 
conditions:
 
\begin{assumption}[Lyapunov Condition]
\label{a-nonint}
There is $v \in \mathcal C^{1,2,(1,1)}(I \times D \times \mathcal{P}_2(D))$, $v\geq 0$, and locally integrable, non-random functions $m_1= m_1(t)$ and $m_2=m_2(t)$ on $I$ such that for any $k\in \bN$, for all $t \in I$, $x\in D_k$ and $\mu \in \mathcal P(D_k)$, we have,	
\begin{equation} 
\label{eq b2}
L^\mu(t,x,\mu) v(t,x,\mu) \leq m_1(t) v(t,x,\mu) + m_2(t).
\end{equation}
\begin{enumerate}[{2.1}a)]
 
\item We say that Lyapunov condition \ref{a-nonint}a holds if \eqref{eq b2} holds and there is a non-negative function $V =  V(t,x)$ such that for any $k\in\bN$, for all $t\in I$, $x\in D_k$ and all $\mu \in \mathcal P(D_k)$, we have, 
\begin{equation}
\label{eq V ineq1}
 V(t,x)\,\leq v(t,x,\mu)\,\  
\end{equation}
and 
\begin{equation}
\label{eq Vk limit}
V_k := \inf_{s \in I, x\in \partial D_k } V(s,x) \,\,\, \text{$\to \infty$ as $k\to \infty$.}	
\end{equation}
\item We say that Lyapunov condition \ref{a-nonint}b holds if \eqref{eq b2} holds and there exists a non-negative function $V$ such that for any $k\in \bN$, for all $t\in I$ and $\mu\in \mathcal P(D_k)$, we have,
\begin{equation}
\label{eq V ineq}
\int_{D_k}  V(t,x)\,\mu(dx) \leq \int_{D_k} v(t,x,\mu)\,\mu(dx) 
\end{equation}
and
\begin{equation}
\label{eq Vk D limit}
V^c_k := \inf_{s \in I, x\in  D_k^c } V(s,x) \,\,\, \text{$\to \infty$ as $k\to \infty$.}	
\end{equation} 
\end{enumerate}
\end{assumption}
\begin{assumption}[Integrated Lyapunov condition]
\label{a-int}
There is a $v \in \mathcal C^{1,2,(1,1)}(I \times D \times \mathcal{P}_2(D))$, $v\geq 0$, such 
that: \\
i) There are locally integrable, non-random, functions $m_1= m_1(t)$ and $m_2=m_2(t)$ on $I$ such that for any $k\in \bN$, for all $t \in I$ and $\mu \in \mathcal P (D_k)$, we have, 
\begin{equation}\label{eq b1}
\int_{D_k} L^\mu (t,x,\mu) v(t,x,\mu) \mu (dx) \leq m_1 (t) \int_{D_k} v(t,x,\mu) \mu (dx) + m_2(t)
\end{equation}
ii) There is a non negative function $V =  V(t,x)$ satisfying \eqref{eq V ineq} and \eqref{eq Vk D limit}.
\end{assumption}

\begin{assumption}[Initial Distribution]\label{ass init}
We assume that for a given Lyapunov function $v$, the initial distribution $\mu_0:=\mathscr L(x_0)$ is such that $\mu_0$ can be approximated by a sequence of probability distributions $(\mu^k_0)_{k}$ such that $\mu^k_0\implies \mu_0$ and for each $k\in\bN $, $\mu^k_0$ is supported on $D_k$ and for some increasing continuous function $\varphi_v:[0,\infty)\rightarrow[0,\infty)$ such that $\varphi_v(x)\geq x$ for all $x\in [0,\infty)$ we have, $$\langle \mu^k_0,v(0,\cdot,\mu^k_0)\rangle \leq \varphi_v(\langle \mu_0,v(0,\cdot,\mu_0)\rangle)<\infty.$$
\end{assumption}

\begin{remark} \label{rmk assumption}
\hfill{  }
\begin{enumerate}[i)]
\item We have deliberately not specified the signs of the functions $m_1$ and $m_2$. 	
 
\item Note that if $\mu_0$ is supported on some $D_K$ for any $K\in \mathbb N$, then Assumption \ref{ass init} is satisfied after relabelling the sequence $(D_k)$ to start from $D_K$ and setting $\mu^k_0=\mu_0$.
\item Regarding Assumption \ref{ass init}, it would be preferable to be able to prescribe an approximating sequence $\mu^k_0$. It is easy to imagine however, how this condition should look in the case where $D=\mathbb R$ and $v(x,\mu):=x^2$. One simply truncates the measure $\mu_0$ on $D_k:=(-k,k)$ and puts the mass of the measure $\mu_0$ outside $D_k$ at the origin i.e. $\mu^k_0(dx):=\1_{x\in D_k}\mu_0(dx)+\mu_0(D_k^c)\delta_{0}$. The increasing continuous function $\varphi_v$ in assumption \ref{ass init} facilitates the finding of such a Lyapunov function. The fact that searching for a Lyapunov function for a McKean-Vlasov SDE should also depend upon the initial distribution and not just the form of the coefficients should not be too surprising given the dependence of the coefficients on the law of the solution. Also, in \cite{bogachevRocknerShaposhnikovStationary2019}, Example 1.1 shows that the existence and convergence to a stationary distribution of the non-linear Fokker Planck equation depends not only on the form of the measure dependence of the coefficients, but on the initial condition. 
\end{enumerate}
\end{remark}

Regarding the continuity of coefficients in~\eqref{eq mkvsde} and their
local boundedness we require the following.
\begin{assumption}[$v$-Continuity]
\label{ass continuity}
Functions $b:I\times D \times \mathcal{P}(D) \to \mathbb{R}^d$ and 
$\sigma:I\times D \times \mathcal{P}(D)\to \mathbb{R}^d\times \mathbb{R}^{d'}$
are jointly continuous in the last two arguments in the following sense:
if $(\mu_n)_n \subset \mathcal P(D)$ are such that
\bothreferees{Supremum in $n$ has been added.}
\[
\sup_n \sup_{t \in I} \int_{ D} v(t,x,\mu_n)\,\mu_n(dx) < \infty
\]
and if $(x_n \rightarrow x, \mu_n \Rightarrow \mu )$ as $n\to \infty$
then for any $t\in I$, $b (t,x_n,\mu_n) \to b(t,x,\mu)$ and $\sigma (t,x_n,\mu_n) \to \sigma(t,x,\mu)$ 
as $n\to \infty$.
\end{assumption}
\secondrefereee{for any $t\in I$ has been added in above assumption}
\begin{assumption}[Local $v$-Boundedness]
\label{ass local boundedness}
There exist constants $c_k \geq 0$ such that for any  $\mu \in \mathcal P(D)$ 
\[
\sup_{x\in D_k} |b(t,x,\mu)| 
\leq c_k \left(1+\int_{D} v(t,y,\mu)\mu(dy)\right)\,,
\]
\[
\sup_{x\in D_k} |\sigma(t,x,\mu)| 
\leq c_k \left(1+\int_{D} v(t,y,\mu)\mu(dy)\right)\,.
\]
\end{assumption}
\begin{assumption}[Integrated $v$-Growth]
\label{ass integrated growth}
There exists an increasing function $\varphi_c$ from $[0,\infty)$ to $[0,\infty)$ such that  for all $\mu \in \mathcal P(D)$, we have,	 	
\[
\int_{D} |b(t,x,\mu)| + |\sigma(t,x,\mu)|^2 \mu(dx) \leq \varphi_c
\left(\int_{D} v(t,x,\mu)\mu(dx)\right),\,\,\,\, \forall t\in I.
\]
\end{assumption}
\bothreferees{Changed in 2 ways, removed the mention of $D_k$ similarly to above and the form of the bound was far from optimal, and has been changed accordingly.}

\color{black}
 
Assumption~\ref{ass continuity} of $v$-continuity in the measure argument is very weak, but may in practice be hard to verify. 
In the case of unbounded domains, the property \eqref{eq V ineq} 
 will often hold for functions of the form $V(x)=|x|^p$, $p \geq 1$. 
In this situation, we have $\mu_n \in \mathcal P_p(D)$ for all the measures $\mu_n$ under consideration for convergence of the coefficients with a uniform bound on their $p^{th}$ moments.
But from~\cite[Theorem 6.9]{villani2009}, we know that for $\mu_n \in \mathcal P_p(D)$ with uniform bound on the $p^{th}$ moments, weak convergence of measures is equivalent to convergence in the $p^{th}$ Wasserstein distance. 
Hence, in such case, it is enough to check that 
if $x_n \rightarrow x$ and  $W_{p}(\mu_n,\mu) \to 0$  as $n\to \infty$
then $b (x_n,\mu_n) \to b(x,\mu)$ and $\sigma (x_n,\mu_n) \to \sigma(x,\mu)$ 
as $n\to \infty$.
This will be satisfied in particular if 
\firstrefereee{Second sentence has been rewritten}
\secondrefereee{With the additional mention of the uniform bound of $p^{th}$ moments, mention of $p'$ can be removed, and has been.}
\begin{equation}\notag
| b(x_n,\mu_n) - b(x,\mu) | + | \sigma(x_n,\mu_n) - \sigma(x,\mu) | \leq \rho(|x-x_n|) +  W_{p}(\mu_n,\mu),
\end{equation}
for some function $\rho = \rho(x)$ such that $\rho(|x|) \to 0$ as $x \to 0$.
We note that this is a common assumption, see e.g.~\cite{funaki1984certain}. 
At this point it may be worth noting that the $p^{th}$-Wasserstein distance on $\mathcal P_p(D)$ is 
\begin{equation}
\label{eq p-wasserstein}
W_p(\mu,\nu) := \left( \inf_{\pi \in \Pi(\mu,\nu)}  \int_{D\times D} |x-y|^p \, \pi(dx,dy)	\right)^{\frac{1}{p}}\,,
\end{equation}
where $\Pi(\mu,\nu)$ denotes the set of {\em couplings} between $\mu$ and $\nu$ i.e. all measures on $\mathscr{B}(D\times D)$ 
such that $\pi(B,D) = \mu(B)$ and $\pi(D, B) = \nu(B)$ 
for every $B \in \mathscr B(D)$.

Note that in the case of McKean--Vlasov SDEs it is often useful to
think of the solution as a pair consisting of the process $x$ and its law
i.e. $(x_t,\caL(x_t))_{t \in I}$.
The coefficients of the McKean--Vlasov SDE depend on the law of the solution 
and the main focus of this paper is on equations with unbounded coefficients, therefore
a condition on integrability of the law is natural. 

\begin{definition}[$v$-Integrable Weak Solution]
\label{def soln}
A $v$-integrable weak solution to~\eqref{eq mkvsde}, on $I$ in $D$ is
\[
\big(\Omega, \mathcal F, \mathbb P, 
(\mathcal F_t)_{t \in I}, (w_t)_{t \in I}, (x_t)_{t \in I}\big),
\]
where $(\Omega, \mathcal F, \mathbb P)$ is a probability space,
$(\mathcal F_t)_{t \in I}$ is a filtration, $(w_t)_{t \in I}$ 
is an $(\mathcal F_t)_{t\in I}$-Wiener process, $(x_t)_{t \in I}$ is an adapted process 
satisfying~\eqref{eq mkvsde} such that $x\in C(I; D)$ a.s.
and finally, for all $t\in I$ we have $\mathbb E v(t,x_t,\mathscr L(x_t)) < \infty$.
\end{definition}

Before we state the main theorem of this paper, we state the conditions on $m_1,m_2$ that allow one to establish the integrability and tightness estimate, which in the case $I=[0,\infty)$ needs to be uniform in time.

\firstrefereee{Remark has been moved to make the definitions of $M$ and $M^+$ part of the main text.  }
Define $\gamma(t) :=  \exp\left(-\int_0^t m_1(s)\,ds\right)$ and
\begin{equation}
\begin{split}
\label{eq: uniform tauk est}
M(t):= &   \frac{\varphi_v(\langle \mu_0,v(0,\cdot,\mu_0)\rangle)}{\gamma(t)} +  \int_{0}^t \frac{\gamma(s)}{\gamma(t)} m_2(s) ds,\\
M^+(t):= & e^{\int_0^t (m_1(s))^+\,ds} \left(\varphi_v(\langle \mu_0,v(0,\cdot,\mu_0)\rangle) +  \int_0^{t } \gamma(s) m_2^+(s)  \,  ds\right)\, .
\end{split}
\end{equation}
\bothreferees{The definitions of $M$ and $M^+$ have been changed to reflect the changes to the initial condition}
Note that $M(t)\leq M^+(t)$. 
\begin{remark}[Conditions on $m_1$ and $m_2$ Ensuring Finiteness of $M^+$]
\label{remark on Mt}
\begin{enumerate}[i)]
	\item If $I=[0,T]$,   $m_1$ and $m_2$ are set to $0$ outside $I$, leading 
to 
\[
\sup_{t<\infty}  \int_{0}^t \frac{\gamma(s)}{\gamma(t)} m_2(s) ds \leq \int_0^{T} e^{\int_s^{T} m_1(r)\,dr} |m_2(s)|ds < \infty\,.
\]
\item 
If $I=[0,\infty)$ and we have 
\begin{equation}
\label{eq extra condition on m1m2}
m_1(t) \leq 0 \,\,\,\forall t\geq0  \,\,\, \text{and}\,\,\, \int_0^\infty  |m_2(s)|\,ds<\infty\,,
\end{equation}
then
\[
 \sup_{t<\infty} \int_0^t e^{\int_s^{t} m_1(r)\,dr} m_2(s)\,ds \leq \int_0^\infty |m_2(s)| \, ds < \infty.
\]
\secondrefereee{$\gamma(s)$ removed from the above integrands for point $ii)$ as requested.}
\end{enumerate}
In both of these cases we have $\sup_{t\in I}M(t)<\infty$ and $\sup_{t\in I}M^+(t) <\infty$. 
\end{remark}

\begin{theorem}\label{thm:weakexistence}
Let $D\subseteq  \mathbb R^d$ and Assumptions~\ref{ass init}, \ref{ass continuity} and~\ref{ass local boundedness} hold. Then the following statements are true:
\begin{enumerate}[i)]
\item If Assumption~\ref{a-nonint}a holds and $\sup_{t\in I } M^+(t) < \infty$, then there exists a $v$-integrable weak solution to~\eqref{eq mkvsde} on $I$.
\item Let either Assumption~\ref{a-nonint}b or Assumption~\ref{a-int} hold. If additionally Assumption~\ref{ass integrated growth} holds and $\sup_{t\in I } M(t) < \infty$, then there exists a $v$-integrable weak solution to~\eqref{eq mkvsde} on $I$.
%
\end{enumerate}
In all of the above cases we also have, 
\[
\sup_{t\in I} \mathbb E v(t, x_t, \mathscr L( x_t)) < \infty\,.
\]
\end{theorem}
\bothreferees{The statement of the theorem has been changed to include the new assumption on the initial distributions and the differing assumptions 2.1a, 2.1b and 2.2}
We make the following comment. 
By virtue of Assumption~\ref{ass local boundedness} we have that under the conditions of Theorem~\ref{thm:weakexistence}, 
the $v$-integrable weak solution to~\eqref{eq mkvsde} obtained by the theorem 
satisfies the forward nonlinear Fokker--Planck--Kolmogorov equation~\eqref{eq fwd kolmogorov}, where $\mu_t = \mathscr L(x_t)$.

\subsection{Proof of the Existence Results}
We will use the convention that the infimum of an empty set is positive infinity.
We extend $b$ and $\sigma$ in a measurable but discontinuous way to functions on $\mathbb{R}^+\times \mathbb R^d \times \mathcal{P}(\mathbb R^d)$ by taking 
\[
b(t,x,\mu) = \sigma(t,x,\mu) = 0 \,\,\, \text{if $x\in \mathbb R^d \setminus D$ or if $t\notin I$}.
\]

For $t\notin I$ we set $m_1(t) = m_2(t)=0$. Consequently from here onwards $I=[0,\infty)$. We define
\[
b^k(t,x,\mu) := \mathds 1_{x\in D_k} b(t,x,\mu)\,\,\,\text{and}\,\,\,
\sigma^k(t,x,\mu) := \mathds 1_{x\in D_k} \sigma(t,x,\mu)\,.
\]

We now provide some results (Lemmas \ref{lemma:uniform} and \ref{lemma tightness} and Corollary \ref{corollary bound for limit process}) regarding a sequence of processes whose existence will be proved as part of Theorem \ref{thm:weakexistence}.

\begin{lemma}
\label{lemma:uniform}
Let Assumptions \ref{ass init} and~\ref{ass local boundedness} hold.
Let there exist a probability space $(\Omega, {\mathcal F}, {\mathbb P})$ with filtration $({\mathcal F}_t)_{t\in I}$, adapted Wiener process $ w$ and adapted processes $(x^k)_k$ that satisfy, for all $t \in I$, 
\begin{equation} \label{eq mkvsde k}
d x^k_t = b^k(t, x^k_t,\mathscr L( x^k_t))\,dt
+ \sigma^k(t, x^k_t,\mathscr L( x^k_t))\,d w_t\,,\,\,\, \mathscr L(x^k_0) =  \mu^k_0\,.
\end{equation}
\firstrefereee{$x^k_0=x_0$ changed to $\mathscr L(x^k_0) =  \mu^k_0$ as in assumption \ref{ass init}}
For any $m,k\in\mathbb N$, let  ${\tau}^k_m := \inf\{t\in I:{x}^k_t \notin D_m\}$.
\firstrefereee{Changed from 'For $m\leq k$, to any $m,k$.}

\begin{enumerate}[i)]
\item If either Assumption~\ref{a-nonint}a, Assumption~\ref{a-nonint}b or~\ref{a-int} hold then for any $t\in I$,
\begin{equation*}
\sup_{k} \mathbb E v(t, x^k_t,\mathscr L( x^k_t))  \leq M(t) \,.
\end{equation*}
\item If either Assumption~\ref{a-nonint}a, Assumption~\ref{a-nonint}b or~\ref{a-int} hold then for any $t\in I$ and $k\in \mathbb N$,
\[
\mathbb P({\tau}_k^k < t) \leq  
 M(t) V_k^{-1}\,.
\]
\bothreferees{Mention of $\mathbb P ({x}_0 \notin D_k)$ has been removed due to the new assumption on the initial condition}
\item  
If Assumption~\ref{a-nonint}a holds then for any $t\in I$,
\[
 \sup_k  \mathbb P({\tau}^k_m < t) \leq M^+(t)V_m^{-1}+\bP(x^k_0\in D_m ).
\]
\item If Assumption~\ref{a-nonint}b or Assumption~\ref{a-int} holds then for any $t\in I$,
\[
\sup_{k}\mathbb P ( x^k_t \notin D_m) \leq  M(t)V_m^{-1}\,.
\]

\end{enumerate}
\color{black}
\end{lemma}

\begin{proof}
i) Since, $\mu^k_0$ is supported on $D_k$ for each $k\in \mathbb N$ and the coefficients $b^k$ and $\sigma^k$ are zero outside $D_k$, we know that the support of $x^k_t$ is contained within $\overline{D}_k$ for all $t\in I$, therefore $\mathscr{L}( x^k_t) \in \cP_2(D)$ and we can apply the It\^o formula from Proposition~\ref{propn ito formula full} to $\gamma v$ with arguments $t$, ${x}^k$ and its law. Thus 
\[
\begin{split}
& \gamma(t) v(t, x^k_t,\mathscr{L}( x^k_t)) =  \gamma(0) v(0, x^k_0,\mathscr{L}( x^k_0)) \\
& \qquad   + \int_0^{t} \gamma(s) [L^\mu v  - m_1 v](s, x^k_s,\mathscr{L}( x^k_s))    \, ds  
+ \int_0^{t} \gamma(s) [(\partial_x v) \sigma](s, x^k_s,\mathscr{L}( x^k_s)) \, d w_s\,.
\end{split}
\]
Due to the local boundedness of the coefficients and either Lyapunov condition~\eqref{eq b2} or~\eqref{eq b1} combined with Remark \ref{rmk assumption} ii) we get
\bothreferees{Added the reference to Remark \ref{rmk assumption}}
\begin{equation}
\label{eq: bound}	
\mathbb E \gamma(t) v(t , x^k_t,\mathscr{L}( x^k_t)) \leq \mathbb E \gamma(0) v(0, x^k_0,\mathscr{L}( x^k_0))
+ \int_0^t \gamma(s) m_2(s)ds \,.
\end{equation}
This proves the first part of the lemma. 

ii) For the second part we proceed as follows. Noting that the coefficients $b^k$ and $\sigma^k$ are zero outside $D_k$, once the process $x^k$ leaves $D_k$ the process stops, yielding ${x}^k_t = {x}^k_{t\wedge  {\tau}_k^k}$ for all $t\in I$, which implies $\mathscr{L}( x^k_t) = \mathscr{L}(x^k_{t\wedge {\tau}^k_k})$ for all $t\in I$. We further observe using \eqref{eq V ineq} that,
\secondrefereee{Justification that ${x}^k_t = {x}^k_{t\wedge  {\tau}_k^k}$ for all $t\in I$ has been added}
\[
\begin{split} 
\mathbb E v(t, x^k_{t},\mathscr{L}( x^k_{t})) =  \,
\mathbb E v(t, x^k_{t\wedge {\tau}_k^k},\mathscr{L}( x^k_{t\wedge{{\tau}}^k_k})) 
\geq \mathbb E [ V(t, x^k_{t\wedge  {\tau}_k^k}) \mathds 1_{ {\tau}_k^k < t} ]   = &  \, \mathbb E [ V(t , x^k_{{\tau}_k^k}) \mathds 1_{   {\tau}_k^k < t} ]\\
\geq  & V_k\mathbb P (   {\tau}_k^k < t)\,.
\end{split} 
\]
Hence,
\[
\begin{split}
& \mathbb P( {\tau}_k^k < t) \leq \frac{\mathbb E v(t, x^k_{t\wedge {\tau}_k^k},\mathscr{L}( x^k_t))}{V_k} \,.	
\end{split}
\]
This completes the proof of the second statement.
\bothreferees{Part ii of the proof was simplified, since we have $\bP(x^k_0\notin D_k)=0$ by the new assumption on the initial condition}

iii) To prove the third statement we first note that for $m>k$ we have $\mathbb P(\tau^k_m < t) = \mathbb P(x^k_0 \notin D_m )=0$. 
Thus we may assume that $m\leq k$.
We proceed similarly as above but with the crucial difference that 
${x}^k_t$ is no longer equal to ${x}^k_{t\wedge \tau^k_m}$. 
Our aim is to apply the It\^{o} formula to the function $v$, the process
$( x^k_{t\wedge \tau^k_m})_{t\in I}$ and the flow of marginal measures $(\mathscr{L}( x^k_t))_{t\in I}$.
Note that $\caL( x^k_{t\wedge \tau^k_m}) \neq \mathscr{L}( x^k_t)$.
Nevertheless the It\^o formula~\ref{propn ito formula full} may be applied.
After taking expectations this yields
\[
\begin{split}
\bE \left[ \gamma(t\wedge \tau^k_m) v(t\wedge \tau^k_m, x^k_{t\wedge \tau^k_m},\mathscr{L}( x^k_t)) \right]
=& \bE v(0, x^k_0,\mathscr{L}( x^k_0))\\
 & +  \bE \int_0^{t\wedge \tau^k_m}  \gamma(s) \left[ L^{\mu,k} v - m_1 v\right](s, x^k_s,\mathscr{L}( x^k_s)) \, ds.  
\end{split}
\]
We now use~\eqref{eq b2} to see that  
\[
\begin{split}
& \bE \left[ \gamma(t\wedge \tau^k_m) v(t\wedge \tau^k_m, x^k_{t\wedge \tau^k_m},\caL( x^k_t)
)\right] 
\leq  \bE v(0, x^k_0,\caL( x^k_0)) + \bE \int_0^{t\wedge \tau^k_m} \gamma(s) m_2(s) \, ds  \\
& \leq  \varphi_v(\langle \mu_0,v(0,\cdot,\mu_0)\rangle) +  \int_0^{t } \gamma(s) m_2^+(s) \,  ds =: \bar M(t)\,.
\end{split}
\]
Then 
\[
\inf_{s\leq t}\gamma(s) \bE  v(t\wedge \tau^k_m,x^k_{t\wedge \tau^k_m},\caL( x^k_t)
) \leq \bE \gamma(t\wedge \tau^k_m) v(t\wedge \tau^k_m,x^k_{t\wedge \tau^k_m},\caL( x^k_t)
) 
\leq \bar M(t)
\]
and so using \eqref{eq V ineq1} we see the following, 
\[
\begin{split}
\E[v(t\wedge\tau^k_m,x^k_{t\wedge\tau^k_m},\mathscr L(x^k_t))]\geq \E[V(t\wedge\tau^k_m,x^k_{t\wedge\tau^k_m})]\geq &
 \E[V(t\wedge\tau^k_m,x^k_{t\wedge\tau^k_m})\1_{\{0<\tau^k_m<t\}}]\\
 \geq & V_m\P(0<\tau^k_m<t).
\end{split}
\]
Combining the above we have,
\[
\begin{split}
 \mathbb P( \tau^k_m < t) = & \P(0<\tau^k_m <t)+ \P(\tau^k_m=0)
\leq   \frac{1}{\inf_{s\leq t} \gamma(s)} \frac{\bar M(t)}{V_m} + \mathbb P ({x}^k_0 \notin D_m)\,.	
\end{split}
\]
\firstrefereee{Indeed, as you remarked in point 19 of your report, this did not follow as stated, however, under the assumption 2.1a, the inequality follows.}
\secondrefereee{$V_m(\mathscr L(x^k_t))$ changed to $V_m$.}

We conclude by observing that 
\[
\inf_{s\leq t} \gamma(s) \geq  e^{-\int_0^t (m_1(s))^+\,ds}.
\]


iv) To prove the fourth statement, first note that for $m>k$, $\bP ( x_t^k \notin D_m)=0$ and hence we take $m\leq k$.
Conditions \eqref{eq V ineq} and \eqref{eq Vk D limit} imply that

\begin{equation*}
\begin{split}
\mathbb E v(t, x^k_{t},\mathscr{L}( x^k_t)) 
\geq & \int_{D} V(t,x)	\mathscr{L}( x^k_t)(dx)
\geq  \int_{D\cap D_m^c} V(t,x)	\mathscr{L}( x^k_t)(dx) 
\geq  V_m \bP(x^k_t \notin D_m)\,.
\end{split}
\end{equation*}

\end{proof}

\bothreferees{Remark label 2.11 has been removed from the following statement, to minimize the change in labelling for the results that follow}
Since we are assuming that $\mathbb P(x_0 \in D) = 1$ we have 
\[
\lim_{k \to \infty} \mathbb P(x_0 \notin D_k) 
= 1 - \lim_{k\to \infty} \mathbb P(x_0 \in D_k) 
= 1 - \mathbb P \big (\bigcup_k \{x_0 \in D_k\}\big ) = 0\,.
\]	


\begin{corollary}
\label{corollary bound for limit process}
Let Assumption~\ref{ass local boundedness} hold.
Let $( \Omega, {\mathcal F},({\mathcal F}_t)_{t\in I}, {\mathbb P})$
be a filtered probability space equipped with an $({\mathcal F}_t)_{t\in I}$-Wiener process $ w$ and a sequence of adapted processes $({x}^k)_k$ such that~\eqref{eq mkvsde k} holds for all $t \in I$, $k\in \mathbb N$.
Assume that $ x^k \to x$ in $C(I; D)$ ${\mathbb P}$-almost surely.	
If either Assumption~\ref{a-nonint}a, \ref{a-nonint}b or~\ref{a-int} hold then \color{black}
\[
\sup_{t\in I} \mathbb E v(t, x_t, \mathscr L( x_t)) \leq \sup_{t\in I} M(t) \,,
\]
where $M$ is given in~\eqref{eq: uniform tauk est}.
\end{corollary}

\begin{proof}
By Fatou's lemma, continuity of $v$ and~\eqref{eq: uniform tauk est} we get
\[
\mathbb E v(t, x_t, \mathscr L( x_t)) 
\leq \liminf_{k\to \infty} \mathbb E v(t, x^k_t, \mathscr L( x^k_t)) 
\leq \sup_{t\in I} M(t)\,.
\]
The results follows if we take supremum over $t$.
\end{proof}

Our aim is to use Skorokhod's arguments to prove the existence of a weak 
(also known as martingale) solution to the equation~\eqref{eq mkvsde}. Before we proceed to the proof of the main theorem, Theorem~\ref{thm:weakexistence}, we need to establish tightness of the family of laws of the processes defined by \eqref{eq mkvsde k}.

\begin{lemma}[Tightness]
\label{lemma tightness}
Let Assumptions \ref{ass init} and \ref{ass local boundedness} hold and let there exist a probability space $(\Omega, {\mathcal F}, {\mathbb P})$ with filtration $({\mathcal F}_t)_{t\in I}$, adapted Wiener process $ w$ and adapted processes $(x^k)_k$ that satisfy, for all $t \in I$, \eqref{eq mkvsde k}.

\begin{enumerate}[i)]
\item If Assumption~\ref{a-nonint}a holds with $\sup_{t\in I } M^+(t) < \infty$, then the law of $( x^k)_k$ is tight on $C(I;  D)$. 
\item  Let Assumptions~\ref{a-nonint}b or~\ref{a-int} hold, with Assumption \ref{ass integrated growth} and $\sup_{t\in I } M (t) < \infty$, then the law of $(x^k)_k$ is tight on $C(I; D)$. \color{black}
Additionally for any $\varepsilon >0$, there is $m_\varepsilon$ such that for all $m\geq m_{\varepsilon}$
\[
\sup_k \bP(\tau^k_m \in I)\leq \varepsilon. 
\]
\end{enumerate}
\end{lemma}
\begin{proof}
$i)$ Under the Assumption~\ref{a-nonint}a tightness of the law of $( x^k)_k$ on $C(I;D)$  follows from the third statement in Lemma~\ref{lemma:uniform}.
Indeed given $\varepsilon > 0$ we can find $m_0$ such that for any $m>m_0$
\[
\sup_{t\in I}\mathbb P(\limsup_{k\to\infty}{\tau}^k_m < t) \leq \sup_{t\in I}\liminf_{k\to \infty}\P(\tau^k_m < t)\leq \sup_{t \in I} M^+(t) V_m^{-1} +\liminf_{k\to \infty}\P(x^k_0\notin D_m) \leq  \varepsilon ,
\]
due to, in particular, our assumption that $V_m \to \infty$ as $m \to \infty$. By considering Assumption \ref{ass local boundedness} of local $v$-boundedness along with the first statement of Lemma $\ref{lemma:uniform}$, the stopped processes $x^k_{\cdot\wedge\tau^k_m}$ are tight in $C([0,t];D)$ for any $t\in I$. Thus, as $\sup_{t\in I}\bP (\limsup_{k\to\infty} \tau^k_m < t )\rightarrow 0$ as $m\rightarrow\infty$, we recover tightness of $x^k$ in $C(I;{D})$.  

$ii)$ First we observe that for every $\ell$ and $(t_1,\ldots,t_{\ell})$ in $I$, the joint distribution of $({x}_{t_1}^k, \ldots, {x}_{t_\ell}^k  )$ is tight. Indeed, 
statement iv) \color{black}in Lemma \ref{lemma:uniform} guarantees tightness of the law of ${x}_{t}^k$ for any $t\in I$. 
Given $\varepsilon > 0$, for any  $\ell \in \bN$  we can find $m_0$ such that for any $m>m_0$
\[
\mathbb P({x}_{t_1}^k\notin D_m, \ldots, {x}_{t_\ell}^k\notin D_m   ) \leq  \ell \sup_{t\in I}M(t) V_m^{-1} \leq \varepsilon\,,
\]
due to, the assumption that $V_m \to \infty$ as $m \to \infty$. 
We will use Skorokhod's Theorem (see~\cite[Ch. 1 Sec. 6]{skorokhod1965}).
This will allow us to conclude tightness of the law of $( x^k)_k$ on $C(I; D)$ as long as we can show that for any $\varepsilon>0$ 
\[
\lim_{h\rightarrow 0} \sup_k \sup_{|s_1 - s_2| \leq h } \bP(| {x}_{s_1}^k - {x}_{s_2}^k   | > \varepsilon ) =0\,.
\]
From~\eqref{eq mkvsde k}, using the Assumption~\ref{ass integrated growth}, we get,
for $0< |s_1 - s_2| < 1$, 
\begin{equation*}
\begin{split}
\bE | x^k_{s_1} - x^k_{s_2} |  \leq  & \int_{s_2}^{s_1} \bE | b^k(r, x^k_r,\mathscr L( x^k_r))|\,dr
+ \left( \bE \int_{s_2}^{s_1} | \sigma^k(r, x^k_r,\mathscr L( x^k_r))|^2 \,dr \right)^{\frac12} \\
 \leq &   c \int_{s_2}^{s_1}   \varphi_c(\sup_k \bE v(r, x^k_r,\mathscr L( x^k_r)) )    \,dr
+ \left( c \int_{s_2}^{s_1}  \varphi_c (  \sup_k \bE v(r, x^k_r,\mathscr L( x^k_r)) ) \,dr \right)^{\tfrac12} \\ 
\leq & c \left(1+\varphi_c(\sup_{t\in I}M(t))\right) |s_1-s_2|^{\tfrac12}  \,.
\end{split}
\end{equation*}
\bothreferees{Last two lines changed to reflect the new integrated growth condition}
Markov's inequality leads to 
\[
\sup_k \sup_{|s_1 - s_2|\leq h} \mathbb P\left(|x^k_{s_1} -  x^k_{s_2}| > \varepsilon\right) \leq c \varepsilon h^{\frac12}
\]
which concludes the proof of tightness. 

We will now prove the second statement in $ii)$.
Note that $C(I;D)$ is open and $C(I;D_{k-1})  \subset C(I;D_{k})$ and $\bigcup_k C(I;D_k ) = C(I;D)$. 
We know that for any $\varepsilon>0$ there is a compact set $\cK_\varepsilon \subset C(I;D)$ such that
\[
\sup_k \bP( x^k \notin \cK_{\epsilon}) \leq \epsilon.
\]

Since $\cK_\varepsilon \subset C(I;D)$  
is compact and the set of $(C(I;D_k))_k$ is an open cover, there must be some $m_\varepsilon$ such that  $\cK_\varepsilon \subset C(I;D_{m_\varepsilon})$.
But this means that 
\[\P(x^k\notin C(I;D_{m_\varepsilon})) \leq \P( x^k \notin \cK_\varepsilon)
\]
and so 
$
\bP(\tau^k_m \in I )= \bP(x^k \notin C(I;D_m)) \leq \bP( x^k \notin C(I;D_{m_\varepsilon})) \leq \P( x^k \notin \cK_\varepsilon)\leq \varepsilon\color{black}
$ 
for all $m\geq m_\varepsilon$.
\bothreferees{This arguement is updated for compact $D_k$}
 \end{proof}

\begin{proof}[Proof of Theorem~\ref{thm:weakexistence}]
Let us define 
$t^n_i := \tfrac{i}{n}$, $i=0,1,\ldots$ and $\kappa_n(t) = t^n_i$ 
for $t\in [t^n_i, t^n_{i+1})$.
Fix $k$.
We introduce Euler approximations $x^{k,n}$, $n\in \mathbb{N}$,
\[
x^{k,n}_t = x_0 + \int_0^t b^k\left(s,x^{k,n}_{\kappa_n(s)},\mathscr L(x^{k,n}_{\kappa_n(s)})\right)\,ds + \int_0^t \sigma^k\left(s,x^{k,n}_{\kappa_n(s)},\mathscr L(x^{k,n}_{\kappa_n(s)})\right)\,dw_s\,.	
\]

Let us outline the proof: As a first step we fix $k$ and we show 
tightness with respect to $n$ and Skorokhod's theorem to take $n\to \infty$. 
The second step is then to use Lemma~\ref{lemma tightness} to show tightness with respect to $k$. 
Finally we can use Skorokhod's theorem again to show that (for a subsequence) the limit 
as $k\to \infty$ satisfies~\eqref{eq mkvsde} (on a new probability space).

{\em First Step.} Using standard arguments, we can verify that, for a fixed $k$,
the sequence $(x^{k,n})_n$ is tight (in the sense that the laws induced
on $C([0,\infty), D)$ are tight).
By Prohorov's theorem (see e.g.~\cite[Ch. 1, Sec. 5]{billingsley}),
there is a subsequence
(which we do not distinguish in notation) 
such that $\mathscr{L}(x^{k,n}) \Rightarrow \mathscr{L}(x^k)$ as $n \to \infty$
(convergence in law).

Hence we may apply Skorokhod's Representation Theorem 
(see e.g. \cite[Ch. 1, Sec. 6]{billingsley})
and obtain a new probability space 
$(\tilde{\Omega}^k, \tilde{\mathcal{F}}^k,\tilde{\mathbb{P}}^k)$ 
where on this space there are new random variables 
$(\tilde x_0^n, \tilde x^{k,n}, \tilde w^n)$ 
and $(\tilde x_0, \tilde x^k, \tilde w)$ such that
\[
\mathscr{L}(\tilde{x}_0^n, \tilde{x}^{k,n}, \tilde{w}^n)
= \mathscr{L}(x_0,x^{k,n},w) \quad \forall n \in \mathbb{N},\quad \quad
\mathscr{L}(\tilde{x}_0, \tilde{x}^{k}, \tilde{w})
=  \mathscr{L}(x_0,x^{k},w)\quad\text{and}
\]
\bothreferees{Sentence removed from this point that stated "After taking another subsequence to obtain almost sure convergence from convergence in probability". It was unnecessary since we cite billingsley rather than skorokhod at this point and thus can claim sure convergence.}
\[
(\tilde{x}_0^n, \tilde{x}^{k,n}, \tilde{w}^n) 
\to (\tilde x_0,\tilde{x}^k, \tilde w) \,\,\, \text{as $n\to \infty$ in $C([0,\infty),D\times  D \times \mathbb R^{d'})$ \text{surely}}.
\]
We let 
\[
\tilde{\mathcal{F}}^k_t := \sigma\{\tilde{x}_0\}\vee \sigma\{\tilde{x}_s, 
\tilde{w}_s : s\leq t\}
\]
and define $\tilde{\mathcal{F}}^{k,n}_t$ analogously. 
Then $\tilde{w}^n$ and $\tilde{w}$ are $(\tilde{\mathcal{F}}^n_t)_{t\geq 0}$
and $(\tilde{\mathcal{F}}_t)_{t\geq 0}$-Wiener processes, respectively.
Define
\[
\tilde{\tau}^{k,n}_k := \inf\{ t \geq 0: \tilde{x}^{k,n}_t \notin D_k \}
\,\,\,\text{and}\,\,\,
 \tilde{\tau}_k^k := \inf\{ t \geq 0: \tilde{x}^k_t \notin D_k \}\,.
\]
These are $\tilde{\mathcal F}^{k,n}$ and $\tilde{\mathcal F}^k$ stopping times respectively.
Moreover, due to the a.s. convergence of the trajectories $\tilde{x}^{k,n}$ to $\tilde{x}^k$ 
we can see that
\[
\liminf_{n\to \infty} \tilde{\tau}^{k,n}_k \geq \tilde{\tau}^k_k\,.
\]
From the fact that the laws of the sequences are identical we see that we
still have the Euler approximation equation on the new probability space: for $t\geq 0$
\[
d\tilde x^{k,n}_t = b^k\left(t,\tilde x^{k,n}_{\kappa_n(t)},\mathscr L(\tilde x^{k,n}_{\kappa_n(t)})\right)\,dt
+ \sigma^k\left(t,\tilde x^{k,n}_{\kappa_n(t)},\mathscr L(\tilde x^{k,n}_{\kappa_n(t)})\right)\,d\tilde w^n_t\,.
\]
 
Using Skorohod's Lemma,
see~\cite[Ch. 2, Sec. 3]{skorokhod1965}, together with 
the continuity conditions in Assumption~\ref{ass continuity}, we can take $n\to \infty$ and conclude
that for all $t\leq \tilde \tau_k^k$ we have
\begin{equation}
\label{eq cut mkvsde}
d\tilde x^k_t = b^k(t,\tilde x^k_t,\mathscr L(\tilde x^k_t))\,dt
+ \sigma^k(t,\tilde x^k_t,\mathscr L(\tilde x^k_t))\,d\tilde w_t\,.
\end{equation}
At this point we remark that the process $\tilde x^k$ stopped at $\tilde \tau_k^k$, is well defined,
continuous on $[0,\infty)$ and satisfies~\eqref{eq cut mkvsde}
for $t\in I$. 
Abusing notation, let $\tilde x^k$ refer to this stopped process. 
Additionally, it satisfies the equation without the cutting applied to the coefficients i.e for all $t\leq \tilde{\tau}^{k}_k$:
\[
d\tilde x^{k}_t = b\left(t,\tilde x^{k}_{t},\mathscr L(\tilde x^{k}_{t})\right)\,dt
+ \sigma\left(t,\tilde x^{k}_{t},\mathscr L(\tilde x^{k}_{t})\right)\,d\tilde w_t\,.
\]
\color{black}

{\em Second Step.}  Tightness of the law of $(\tilde x^k)_k$ in $C(I; \bar D)$ follows from Lemma \ref{lemma tightness} and Remark~\ref{remark on Mt}.
From Prohorov's theorem we thus get that for a subsequence
$\mathscr L(\tilde x^k) \Rightarrow \mathscr L(\tilde x)$ as $k\to \infty$
(convergence in law).
From Skorokhod's Representation Theorem 
we then obtain a new probability space 
$(\bar \Omega, \bar{\mathcal{F}},\bar{\mathbb{P}})$ 
carrying new random variables 
$(\bar x_0^k, \bar x^k, \bar w^k)$ 
and $(\bar x_0, \bar x, \bar w)$ such that
\[
\mathscr{L}(\bar x_0, \bar x, \bar w ) 
= \mathscr{L}(\tilde{x}_0, \tilde{x}, \tilde{w})\,,
\]
\[
\mathscr{L}(\bar x_0^k, \bar x^k, \bar w^k) 
= \mathscr{L}(\tilde{x}_0^k, \tilde{x}^k, \tilde{w}^k )
\quad \forall k \in \mathbb{N},
\]
and 
\[
(\bar{x}_0^k, \bar{x}^k, \bar{w}^k) 
\to (\bar x_0,\bar{x}, \bar w) \,\,\, \text{as $k\to \infty$ in $C( I ; D \times \bar D \times \mathbb{R}^{d'})$ surely}\,.
\]
Let $\bar \tau_k^k := \inf\{ t: \bar x^k_t \notin D_k \}$,
$\bar \tau^k_m := \inf\{ t: \bar x^k_t \notin D_m \}$
and $\bar \tau^{\infty}_m := \inf\{ t: \bar x_t \notin D_m \}$.
Since
$\sup_{t < \infty} |\bar x^k_t - \bar x_t| \to 0$ we get $\limsup_{k\to\infty}\bar \tau^k_{m-1} \leq  \bar \tau^{\infty}_m$ surely. To see why this holds, assume the contrary for finite $\bar \tau^\infty_m(\omega)$ since the infinite case holds immediately. We assume for a contradiction that $\limsup_{k\to\infty} \bar \tau^k_{m-1}(\omega)> \bar \tau^\infty_m(\omega)$.
Then, there exists a subsequence $k_j$ such that $\bar \tau^{k_j}_{m-1}>\bar \tau^\infty_m$ for all $j\in\bN$. Consequently, $|\bar x^{k_j}_{\bar \tau^\infty_m}- \bar x_{\bar \tau^\infty_m}|\geq d(D_{m-1}, \partial D_{m})>0$. However, $|x^{k_j}_{\bar \tau^\infty_m}- \bar x_{\bar \tau^\infty_m}|\leq \sup_{t < \infty}|\bar x^{k_j}_t-\bar x_t|\to 0 $ as $k_j\to\infty$ and we have arrived at a contradiction. 
Then from Fatou's Lemma, and either part iii) of Lemma \ref{lemma:uniform} or part ii) of Lemma \ref{lemma tightness} depending on the type of Lyapunov condition that holds, we have that, for any $s,t\in I$, $t<s$,
\begin{equation}\label{eq tightness limit}
\begin{split}
 \mathbb P(\bar \tau^{\infty}_m \leq t) \leq \P((\limsup_{k\to\infty}\bar \tau^k_{m-1}) < s) \leq & \P(\liminf_{k\to\infty}\{\bar \tau^k_{m-1} < s\})\\
\leq & \liminf_{k\to \infty} \mathbb P(\bar \tau^k_{m-1} < s ) 
\leq \sup_k \mathbb P(\bar \tau^k_{m-1} \in I)    \to 0 \,\,\, \text{as $m \to \infty$.}
\end{split}
\end{equation}
Then the distribution of $\bar \tau^{\infty}_m$ converges in distribution, as $m\to \infty$, to a random 
variable $\bar \tau$ with distribution $\mathbb P(\bar \tau \leq T) = 0$ for any $T<\infty$ and $\mathbb P(\bar \tau = \infty) = 1$. 
In general, convergence in distribution does not imply convergence in probability.
But in the special case that the limiting distribution corresponds to 
a random variable taking a single value a.s. we obtain convergence in 
probability (see e.g.~\cite[Ch. 11, Sec. 1]{dudleybook}).
Hence $\bar \tau^{\infty}_m \to \infty$ in probability as $m\to \infty$. 
From this we can conclude that there is a subsequence that
converges almost surely.

Since~\eqref{eq cut mkvsde} holds for $\tilde x^k$ we have the corresponding
equation for $\bar x^k$ i.e.
for $t\leq \bar \tau^k_k$,
\begin{equation}
\label{eq cut mkvsde bar}
d\bar x^k_t = b(t,\bar x^k_t,\mathscr L(\bar x^k_t))\,dt
+ \sigma(t,\bar x^k_t,\mathscr L(\bar x^k_t))\,d\bar w^k_t\,.
\end{equation}
Fix $m < k'$. 
We will consider $k > k'$.
Then~\eqref{eq cut mkvsde bar} holds for all 
$t\leq \inf_{k\geq k'} \bar \tau^k_m$. 
We can now consider $\bar x^k_{t\wedge \tau^k_m}$ 
(these all stay inside $D_m$ for all $k > k' > m$) 
and use dominated convergence theorem for the bounded variation integral and Skorokhod's lemma on convergence of stochastic integrals, see~\cite[Ch. 2, Sec. 3]{skorokhod1965}, and our assumptions
on continuity of $b$ and $\sigma$ to let $k\to \infty$.
We thus obtain, for $t\leq  \inf_{k\geq k'} \bar \tau^k_m\wedge \bar\tau^\infty_m$,
\begin{equation}
\label{eq mkvsde bar}
d\bar x_t = b(t,\bar x_t,\mathscr L(\bar x_t))\,dt
+ \sigma(t,\bar x_t,\mathscr L(\bar x_t))\,d\bar w_t\,.
\end{equation}
Now, for each fixed $m<k'$,
\[
\lim_{k'\to \infty} \inf_{k\geq k'} \bar \tau^k_m = \lim_{k\to\infty} \bar \tau^k_m = \bar \tau^{\infty}_m.
\]
Finally we take $m \to \infty$ and since $\bar \tau^{\infty}_m \to \infty$ 
we can conclude that~\eqref{eq mkvsde bar} holds for all $t\in I$.
The last statement of the theorem follows from  Corollary~\ref{corollary bound for limit process}.
\end{proof}
\color{black}

\subsection{Examples}
\begin{example}[Integrated Lyapunov condition]
\label{example integrated lyapunov}
Consider the McKean--Vlasov stochastic differential equation~\eqref{eq example1}
i.e.
\[
	d x_t  = -x_t \bigg[\int_\bR y^4 \mathscr L(x_t)(dy)\bigg]\, dt + \frac{1}{\sqrt{2}} x_t\,  d w_t\,,\,\,\,
	 x_0 = \xi >0\,.
\]
Then for $v(x) =  x^4$ we have,
\[
L(x,\mu) v(x) = 3 x^4 - 4 x^4 \int_{\mathbb R} y^4 \, \mu(dy)\,.
\]
We see that the stronger Lyapunov condition~\eqref{eq b2} will not hold with $m_1 < 0$
(at least for chosen $v$, which seems to be a natural choice) and $D_k=(-k,k)$.
However, integrating leads to
\[
\int_{\bR}L(x,\mu) v(x)\mu(dx) 
= 3 \int_\bR x^4 \mu(dx) - 4 \bigg(\int_\bR x^4 \mu(dx)\bigg)^2
\]
using this we will show that the integrated Lyapunov condition~\eqref{eq b1} holds i.e. that 
\[
\int_{\bR}L(x,\mu) v(x)\mu(dx) 
\leq -\int_\bR v(x) \mu(dx) + 1
\]
is satisfied.
To see this we note that $3a-4a^2\leq 1-a$ since $-1+4a-4a^2 \leq - (1-2a)^2$. Moreover, Assumption \ref{ass integrated growth} is satisfied.  Condition~\eqref{eq lyapunov example} allows us to obtain 
uniform-in-time integrability properties for $(x_t)$ needed to study 
e.g. ergodic properties. 
\end{example}

\begin{example}[Non-linear dependence of measure and integrated Lyapunov condition] 
\label{example 2}
Consider the McKean--Vlasov stochastic differential equation~\eqref{eq example2} i.e.
\[
dx_t = - \left( \int_\bR (x_t - \a y)\caL{(x_t)}(dy) \right)^3dt + \left( \int_\bR (x_t - \a y)\caL{(x_t)}(dy) \right)^2 \sigma \, dw_t\,, 	
\]
for $t\in I$ and with $x_0 \in L^4(\cF_0,\mathbb R)$.
Assume that $m:= -(6\sigma^2 - 4 + 4\alpha) >0$.
The diffusion generator given by~\eqref{eq Lmu} is
\[
\begin{split}
& (L^\mu v)(x,\mu) =  \bigg( 
 \frac{\sigma^2}{2}\left( \int_{\bR} (x - \a y)\mu(dy)  \right)^4 \partial_x^2  v 
- \left(  \int_{\bR} (x - \a y)\mu(dy)  \right)^3 \partial_x v\bigg)(x,\mu) \\
& + \int_{\bR}\left(   \frac{\sigma^2}{2} \left( \int_{\bR} (z - \a y)\mu(dy)  \right)^4 (\partial_z \partial_\mu v)(t,x,\mu)(z)  
- \left(  \int_{\bR} (z - \a y)\mu(dy)  \right)^3 (\partial_\mu v) (t,x,\mu)(z) 
   \right) \mu(dz)\,. 
\end{split}
\]
We will show that for the Lyapunov function
\[
v(x,\mu) = \left( \int_{\bR} (x - \a y)\mu(dy)  \right)^4\,,
\]   
we have 
\[
\int_{\bR}(L^\mu v)(x,\mu)\,\mu(dx) \leq m  - m  \int_{\bR}v(x,\mu)\,\mu(dx)\,. 
\]
Indeed,
\[ 
\quad \, \,\partial_x	v(x,\mu) = 4 \left( \int_{\bR} (x - \a y)\mu(dy)  \right)^3\,, \,\,\,\,
\partial^2_x v(x,\mu) = 12\left( \int_{\bR} (x - \a y)\mu(dy)  \right)^2\, \,,
\]
\[
\partial_\mu v(x,\mu)(z) = -4\a \left( \int_{\bR} (x - \alpha y)\mu(dy)  \right)^3 \,\,\,\text{and} \,\,\,
\partial_z \partial_\mu v(x,\mu)(z) = 0 \,. 
\]
Hence 
\[
\begin{split}
\,\,(L^{\mu}v)(x,\mu) = & (6\sigma^2-4) \left(\int_\bR (x-  \alpha y) \mu(dy)\right)^6 \\
& + 4\alpha \int_\bR \left[ \left(\int_\bR(z-\alpha y)\mu(dy)\right)^3\left(\int_\bR (x-\alpha y)\mu(dy)\right)^3 \right]\,\mu(dz)\,.
\end{split}
\]
Since we want an estimate over the integral of the diffusion generator we
observe that 
\[
\begin{split}
I   :=  & \int_\bR \int_\bR \left[ \left(\int_\bR(z-\alpha y)\,\mu(dy)\right)^3 \left(\int_\bR (x-\alpha y)\mu(dy)\right)^3 \right]\,\mu(dz)\,\mu(dx)\\	
 \leq &   \left(\int_\bR \left|\int_\bR(z-\alpha y)\,\mu(dy)\right|^3 \,\mu(dz)\right)^2\,.
\end{split}
\]
By the Cauchy--Schwarz inequality we obtain
\[
I \leq  \int_\bR \left(\int_\bR(x-\alpha y)\,\mu(dy)\right)^6 \,\mu(dx) \,. 
\]
Hence, recalling $m:= -(6\sigma^2 - 4 + 4\alpha) > 0$ 
and using the inequality $-x^6 \leq 1-x^4$, we obtain that  
\[
\int_\bR (L^\mu v)(x,\mu)\,\mu(dx) 
\leq \int_\bR (6\sigma^2-4+4\alpha) \left(\int_\bR (x-\alpha y)\, \mu(dy)\right)^6   \mu(dx)
\leq m - m  \int_{\bR}v(x,\mu)\,\mu(dx)\,. 
\]
Furthermore, Assumption \ref{ass integrated growth} is readily satisfied.
\end{example}
\begin{example}[Dependence on measure not through its moments] 
\label{example 3}
Let $\mu$ be a law on $(\mathbb R, \mathscr B(\mathbb R))$ and 
let $F^{-1}_\mu:[0,1]\rightarrow \bR$ be the generalized inverse cumulative distribution function for this law.
Recall that the $\alpha$-Quantile is given by
\[
F_\mu^{-1}(\alpha):=\inf\{x\in \bR\,:\, \mu((-\infty, x]) \geq \alpha]\}\,.
\]
Define the Expected Shortfall of $\mu$ at level $\alpha$, $ES_\mu(\alpha)$, as 
\[
ES_\mu(\alpha):=\frac{1}{\alpha}\int_0^\alpha F_\mu^{-1}(s)\,ds\,.
\]
It is easy to see that for fixed $\alpha$, Expected Shortfall is a Lipschitz continuous function of measure w.r.t $p$-th Wasserstein distances for $p\geq 1$. 
Indeed fix $\mu$, $\nu \in \mathcal P_p(\mathbb R)$ and observe that
\[
\begin{split}
\left|ES_\mu(\alpha)-ES_\nu(\alpha)\right|&\leq\frac{1}{\alpha}\int_0^\alpha |F_\mu^{-1}(s)-F^{-1}_\nu(s)|\, ds\\
&\leq \frac{1}{\alpha}\int_0^1 |F_\mu^{-1}(s)-F^{-1}_\nu(s)|\,ds=\frac{1}{\alpha}W_1(\mu,\nu)\leq \frac{1}{\alpha}W_p(\mu,\nu)\,,	
\end{split}
\]
where the equality above follows from~\cite{vallender1972calculation} and an obvious geometric consideration.

We consider the following one-dimensional example, based loosely on transformed CIR:
\[
d x_t = \frac{\kappa}{2}\big[((ES_{\mathscr{L}(x_t)}(\alpha)\vee \theta)-\frac{\sigma^2}{4\kappa})   x_t^{-1}-x_t\big]\, dt +\frac{1}{2}\sigma \,dw_t.
\]
Here $x_0$ satisfies $\P[x_0>0]=1$ and $\kappa \theta \geq \sigma^2$. 

Note that by defining $D:=(0,\infty)$ and $D_k:=(\frac{1}{k},k)$, we have boundedness of the coefficients on $D_k$ and from the above observations and assumptions one can easily verify that the conditions of Theorem \ref{thm:weakexistence} are satisfied. In particular consider $v(x)=x^2+x^{-2}$. Then,

\begin{equation}\notag
\begin{alignedat}
{1}
L(x,\mu)v(x) &= \frac{\kappa}{2}\big[((ES_{\mu}(\alpha)\vee \theta)-\frac{\sigma^2}{4\kappa})   x^{-1}-x\big](2(x-x^{-3}))+\frac{1}{8}\sigma^2(2+6x^{-4}) \\
&={\kappa}\big[((ES_{\mu}(\alpha)\vee \theta)-\frac{\sigma^2}{4\kappa})\big] -\kappa x^2 - \big[\kappa (ES_{\mu}(\alpha)\vee \theta)- {\sigma^2}  \big]x^{-4}+\kappa x^{-2}+\frac{\sigma^2}{4}\\
&\leq {|\kappa|}  |ES_{\mu}(\alpha)|+ \kappa \theta  +\kappa x^{-2}\\
&\leq {|\kappa|}  |\int_\bR x\,\mu(dx)|+ \kappa \theta  +\kappa x^{-2} \\
 &\leq \frac{1}{2}{\kappa^2} +   \frac{1}{2}\int_\bR x^2\,\mu(dx) + \kappa \theta  +\kappa x^{-2}.\\
\end{alignedat}
\end{equation}
Integrating with respect to $\mu$ we see that condition \eqref{eq b1} holds. Therefore, due to Theorem~\ref{thm:weakexistence}, we have existence of a weak solution to the above McKean--Vlasov equation.
\end{example}

\section{Uniqueness}\label{sec uniq}

In this Section we prove continuous dependence on initial conditions and uniqueness under two types of Lyapunov conditions. 
For the novel {\em integrated} global Lyapunov condition we provide
an example that has been inspired by the work of~\cite{Scheutzow87} on non-uniqueness of solutions to McKean--Vlasov SDEs.

\subsection{Assumptions and Results}\label{subsec uniq results}
Recall that by $\pi \in \Pi(\mu, \nu)$ we denote a coupling between measures $\mu$ and $\nu$. 
In this section we work with a subclass of Lyapunov functions $\bar v\in C^{2}( \mathbb R^d)$ that has the properties: $\bar v\geq 0$, $\text{Ker}\,\, \bar  v= \{0\}$ and 
for all $z\in \mathbb R^d$ we have $\bar v(z) = \bar v(-z)$. 
For this class of Lyapunov functions  we define Wasserstein semi-distance on $\cP(D)$ as\color{black},
\begin{equation}
\label{eq v-wasserstein}
W_{ \bar v}(\mu,\nu) := \left( \inf_{\pi \in \Pi(\mu,\nu)}  \int_{D\times D} \bar v(x-y) \, \pi(dx,dy)	\right)\,.
\end{equation}
Indeed $W_v$ is a semi-metric and the triangle inequality, in general, does not hold. Note that $\bar v$ does not depend on a measure. For $(t,x,y)\in I \times D \times D$, $(\mu,\nu) \in \mathcal P(D) \times \mathcal P(D)$, 
and any $\varphi\in C^2(\mathbb R^d)$, we define
\begin{align*}
L(t,x,y,\mu,\nu)\varphi(x-y) := &   ( b(t,x,\mu) - b(t,y,\nu)) \partial_x \varphi(x-y)\\
 & + \frac{1}{2}\text{tr}\big((\sigma(t,x,\mu) - \sigma(t,y,\nu))(\sigma(t,x,\mu) - \sigma(t,y,\nu))^*  \partial_x^2 \varphi(x-y)\big)  \,.
\end{align*}
 \secondrefereee{Mention of 'generator' removed here for simplicity, as we are just defining notation. $\varphi$ is used as the test function now.}

\bothreferees{The following assumptions were a little too strong, since in practice, one needs finiteness of say the moments appearing in coefficients in order to establish either of the following two conditions. So now the assumptions are formulated only for solution measures instead of all measures.}

\begin{assumption}[Global Lyapunov condition] 
\label{ass gmc}
There exists $\bar v\in C^{2}( \mathbb R^d)$ and locally integrable, non-random, functions $g=g(t)$ and  $h=h(t)$ on $I$, 
such that for any two solutions $(x_t)_{t\in I}$ and $(y_t)_{t\in I}$ to \eqref{eq mkvsde}, with $\mathscr L(x_t)=\mu_t$ and $\mathscr L(y_t)=\nu_t$, for all $t\in I$, 
\begin{equation}\label{eq gmc}
L(t,x_t,y_t,\mu_t,\nu_t)\bar v( x_t-y_t) \leq  g(t) \bar v(x_t-y_t) + h(t) W_{\bar v}(\mu_t,\nu_t)\,.
\end{equation}
\end{assumption}

\begin{assumption}[Integrated Global Lyapunov condition] 
\label{ass int gmc}
There exists $\bar v\in C^{2}( \mathbb R^d)$ and a locally integrable, non-random function $h=h(t)$ on $I$, such that for any two solutions $(x_t)_{t\in I}$ and $(y_t)_{t\in I}$ to \eqref{eq mkvsde}, with $\mathscr L(x_t)=\mu_t$ and $\mathscr L(y_t)=\nu_t$, for all $t\in I$,
and for all couplings $\pi \in \Pi(\mu_t,\nu_t)$
\begin{equation}
\label{eq igmc}	
\int_{D \times D }L(t,p,q,\mu_t,\nu_t)\bar v(p-q)  \pi(dp, dq) 
\leq  h(t) \int_{D\times D} \bar v(p-q) \, \pi(dp,dq)\,.
\end{equation}
\end{assumption}
It can be shown that the three examples given in the previous section satisfy Assumption \ref{ass int gmc} with $\bar v(z)=z^2$.
 
Theorem~\ref{thm contdep} gives a stability estimate for the solution to \eqref{eq mkvsde} with respect to initial condition 
(continuous dependence
on the initial conditions). 

\begin{theorem}[Continuous Dependence on Initial Condition]\label{thm:pathuniq}
\label{thm contdep}
Let Assumption~\ref{ass local boundedness} hold. 
Let $x^i$, $i=1,2$ be two solutions to~\eqref{eq mkvsde} on the  same probability space such that $\bE \bar v(x^1_0 - x^2_0)<\infty$.
\begin{enumerate}[i)]
\item If Assumption~\ref{ass gmc} holds then for all $t\in I$	
\begin{equation}
\label{eq cont dep}
\bE \bar v( x^1_t - x^2_t)  \leq \exp\left( \int_0^t \left[g(s) + h(s) + |h(s)|\right] \,ds \right)  \bE \bar v(x^1_0 - x^2_0)\,.	
\end{equation}
\secondrefereee{factor of 2 in front of $|h|$ removed.}
\item If Assumptions~\ref{ass int gmc} and either Assumption~\ref{a-nonint}a, \ref{a-nonint}b or~\ref{a-int} hold and if there are $p,q$ with $1/p+1/q = 1$ and a constant $\kappa$ such that
for all $(t,x,\mu)$ in $I \times D \times \cP(D)$
\begin{equation}
\label{eq integrability for uniqueness}
|\partial_x \bar v(x -y)|^{2p} + |\sigma(t,x,\mu)|^{2q} + |\sigma(t,y,\nu)|^{2q} \leq \kappa (1 + v(t,x,\mu)+v(t,y,\nu))  	
\end{equation}
then for all $t\in I$
\begin{equation}
\label{eq int cont dep}
\mathbb  E \bar v(x^1_t - x^2_t) \leq \exp\left( \int_0^t  h(s) \,ds  \right) \mathbb E \bar v(x^1_0 - x^2_0)\,.	
\end{equation}
\end{enumerate}
\end{theorem}
First we note that in the case when $I$ is a finite time interval then the sign
of the functions $g$ and $h$ plays no significant role. 
In relation to the study of ergodic SDEs e.g.~(18) in~\cite{Bolley2007} we make the following observations.
If $I=[0,\infty)$ and Assumption~\ref{ass gmc} holds with $g + h + |h| < 0$ then $\lim_{t\to \infty} \mathbb E \bar v(x^1_t - x^2_t) = 0$. 
However we see that while the spatial dependence of coefficients can
play a positive role for the stability of the equation (if $g$ is negative) it seems that the measure dependence never has such positive
role, regardless of the sign of $h$.
If $I=[0,\infty)$ and we are in the second case of Theorem~\ref{thm contdep} 
then negative $h$ can play a positive role for stability (but unlike the 
first case we also need the condition~\eqref{eq integrability for uniqueness}).

\begin{proof}
Note that if we are in case ii) then, in the following we set $g = 0$ for all $t\in I$.
Let
\[
\varphi(t) = \exp\left(-\int_0^t [g(s) + h(s)]\, ds\right) \,.
\]
Applying the classical It\^o formula to $\varphi\, \bar v(x^1-x^2)$ we have that for $t\in I$
\begin{equation}
\label{eq after ito for uniq}
\begin{split}
\varphi(t) \bar v(x^1_t-&x^2_t)  = \bar v(x^1_0 - x^2_0) \\
& +  \int_0^{t} \varphi(s) \big[
L(s,x^1_s,x^2_s,\mathscr L(x^1_s),\mathscr L(x^2_s))\bar v(t,x^2_s-x^2_s)
- (g(s) + h(s)) \bar v(x^1_{s}-x^2_{s}) \big]\,ds\\  
  & +  \int_0^{t} \varphi(s)   \partial_x \bar  v( x^1_s-x^2_s)(\sigma(s,x^1_s,\mathscr L(x^1_s))-\sigma(s,x^2_s,\mathscr L(x^2_s))) dw_s.
\end{split}
\end{equation}
{\em Case i)}
Assumption~\ref{ass gmc} implies
\begin{equation*}
\begin{alignedat}{1}
\varphi(t) \bar v(x^1_{t}-x^2_{t}) \leq \, & \bar v(x^1_0 - x^2_0)
+ \int_0^t \varphi(s)\big[ h(s) W_{\bar v}(\mathscr L(x^1_s),\mathscr L(x^2_s))
 - h(s) \bar v(x^1_{s}-x^2_{s}) \big]
 \,ds\\  
 & +   \int_0^{t} \varphi(s) \partial_x  \bar v( x^1_s-x^2_s)(\sigma(s,x^1_s,\mathscr L(x^1_s))-\sigma(s,x^2_s,\mathscr L(x^2_s))) dw_s.
\end{alignedat}
\end{equation*}
Define the stopping times $\{\tau^i_m\}_{m\geq 1}$, $i=1,2$ and $\{\tau_m\}_{m\geq 1}$
\[
\tau^i_m := \inf \{t\in I \,:\, x^i_t\notin D_m  \}\,,\,\, i = 1,2\,\,\,\text{and}\,\,\, \tau_m := \tau^1_m \wedge \tau^2_m\,.
\]
By Definition~\ref{def soln} we know that $x^i \in C(I;D)$ a.s. 
and so $\tau^i_m \nearrow \infty$ a.s. and hence $\tau_m \nearrow \infty$ a.s. as $m\to \infty$.     
The local boundedness  of $\sigma$ ensures that the stochastic integral in the above is a martingale on $[t\wedge\tau_m]$, hence
\begin{equation*}
\begin{alignedat}{1}
& \bE[ \varphi(t\wedge\tau_m) \bar v( x^1_{t\wedge\tau_m}-x^2_{t\wedge\tau_m})]\\ 
& \leq \bE[\bar v(x^1_0 - x^2_0)] 
+ \bE\left[\int_0^{t\wedge\tau_m} \varphi(s)\big[ h(s)W_{\bar v}(\mathscr L(x^1_s),\mathscr L(x^2_s))
 - h(s) \bar  v( x^1_{s}-x^2_{s}) \big]
 \,ds \right]\\  
& \leq  \bE[\bar v( x^1_0 - x^2_0 )] +  \bE\left[\int_0^{t} \varphi(s)\big[ |h(s)| \bar v(x^1_{s}-x^2_{s}) \big]\,ds \right]
 \,,
\end{alignedat}
\end{equation*}
where the last inequality follows from the definition of the semi-Wasserstein distance. 
Since $\tau_m\nearrow\infty$ as $m\rightarrow\infty$, application of Fatou's Lemma gives 
\[
\bE[\varphi(t)\bar v( x^1_t - x^2_t )] \leq \E \bar v(x^1_0 - x^2_0) + \int_0^t |h(s)|\E[ \varphi(s) \bar v( x^1_s-x^2_s )]\,ds.
\]
From Gronwall's lemma we get~\eqref{eq cont dep}. 

{\em Case ii)} 
Taking expectation in~\eqref{eq after ito for uniq}, recalling that in this case $g=0$ and then using Assumption~\ref{ass int gmc} we have
\begin{equation*}
\mathbb E \left[\varphi(t) \bar v (x^1_{t}-x^2_{t}) \right] 
\leq  \mathbb E\bar v(x^1_0 - x^2_0) 
+  \mathbb E\int_0^{t} \varphi(s) \partial_x \bar v( x^1_s-x^2_s)(\sigma(s,x^1_s,\mathscr L(x^1_s))-\sigma(s,x^2_s,\mathscr L(x^2_s)))\, dw_s.
\end{equation*}
\secondrefereee{Assumption \ref{ass local boundedness} is claimed in the theorem, allowing application of corollary \ref{corollary bound for limit process}}
Corollary~\ref{corollary bound for limit process} together with~\eqref{eq integrability for uniqueness} 
and local integrability of $g$ and $h$
ensures that stochastic integral in the above expression is a martingale. 
Indeed
\[
\begin{split}
& \int_0^t \varphi(s)^2\mathbb E\left[ |\partial_x  \bar v(x^1_s - x^2_s)|^2 |\sigma(s,x^1_s,\mathscr L(x^1_s))-\sigma(s,x^2_s,\mathscr L(x^2_s))|^2 \right]\,ds \\
& \leq \int_0^t \varphi(s)^2\mathbb E\left[ \frac{1}{p}|\partial_x \bar v(x^1_s - x^2_s)|^{2p} + \frac{1}{q}|\sigma(s,x^1_s,\mathscr L(x^1_s))-\sigma(s,x^2_s,\mathscr L(x^2_s))|^{2q} \right]\,ds \\	
& \leq c_{p,q} \int_0^t \varphi(s)^2\mathbb E\left[\partial_x  \bar v(x^1_s - x^2_s)|^{2p} 
+ |\sigma(s,x^1_s,\mathscr L(x^1_s))|^{2q}+|\sigma(s,x^2_s,\mathscr L(x^2_s))|^{2q} \right]\,ds \\	
& \leq c_{p,q} \int_0^t \varphi(s)^2 \kappa \left( 1 +  \mathbb E v(s,x^1_s,\mathscr L(x^1_s)) + \mathbb Ev(s,x^2_s,\mathscr L(x^2_s)) \right)\,ds < \infty\,.
\end{split}
\]
Hence
\begin{equation*}
\varphi(t)\bE[ \bar v(x^1_{t}-x^2_{t})] \leq \bE[ \bar v(x^1_0 - x^2_0)]\,. 
\end{equation*}
\end{proof}

\begin{corollary}
\label{corollary uniqueness}
Let the conditions for either case i) or ii) of Theorem~\ref{thm contdep} hold with either \newline 
$\sup_{t\in I}\exp\left( \int_0^t[g(s)+h(s)+|h(s)|]\,ds \right)<\infty $ or $\sup_{t\in I}\exp\left( \int_0^th(s)\,ds\right)<\infty $ respectively. If $x^1_0 = x^2_0$ a.s. then the solutions to~\eqref{eq mkvsde} are pathwise unique.	
\end{corollary}
\begin{proof}
Since $\text{Ker}\,\,\bar v = \{ 0 \}$, we have that for all $t\in I$, $\P(x^1_t=x^2_t)= 1$. Then, since the processes have continuous paths, we can conclude that they are indistinguishable. 
\end{proof}
\secondrefereee{Indeed, we should have assumed the added uniform bound.}
\subsection{Example due to Scheutzow.}

Consider the McKean--Vlasov SDE of the form 
\begin{equation}
\label{eq:examplescheutzow}
x_t=x_0+\int_0^t B(x_s,\E[\bar b(x_s)])\,ds + \int_0^t \Sigma(x_s,\E[\bar \sigma(x_s)])\,dw_s\,.
\end{equation}
Our study of this more specific form of McKean--Vlasov SDE 
is inspired by~\cite{Scheutzow87}, where it has been shown that in the case 
when $\Sigma=0$ and either of functions $B$ or $\bar b$ is only locally Lipschitz continuous then uniqueness, in general, does not hold. 
We will show that if we impose some structure on the local behaviour of the functions then these, together with the integrability conditions established in Theorem~\ref{thm:weakexistence}, are enough to obtain unique solution~\eqref{eq:examplescheutzow}. 
To be more specific: we impose local (in the second variable) monotone condition on functions $B$ and $\Sigma$, which is weaker than local (in the second variable) Lipschitz condition, and local Lipschitz condition on functions $\bar b$ and $\bar \sigma$. 
	
\begin{assumption} \label{ass scheutzow}
\hfill{}
\begin{enumerate}[i)]
\item Local Monotone condition: there exists locally bounded function $M=M(x',y',x'',y'')$ such that $\forall x,x',x'',y,y',y'' \in D$
\[
\begin{split}
2(x-y)(B(x,x')-B(y,y')) + &|\Sigma(x,x'') -  \Sigma(y,y'')|^2 \\
&\leq  M(x',y',x'',y'')(| x - y |^2 + | x' - y' |^2 + |x'' - y''|^2) 
\end{split}
\]
\hspace{-6mm}There exists a constant $\kappa$ such that:
\item 
$\forall (t,x,\mu) \in I \times D \times \cP(D)$    $|\bar b(x)| + |\bar \sigma (x)| \leq \kappa(1 + v(t,x,\mu) )$, and
\item $\forall (t,x,y,\mu) \in I \times D\times D \times \cP(D)$
\secondrefereee{specified for all $y\in D$.}
\[
|\bar b(x)-\bar b(y)|+ |\bar \sigma(x)-\bar \sigma(y)|\leq \kappa(1+\sqrt{v(t,x,\mu) } + \sqrt{v(t,y,\mu)})|x-y|\,.
\]
\end{enumerate}
\end{assumption}

\begin{theorem}
If Assumptions~\ref{a-int} hold,
if $\sup_{t\in I} M(t) < \infty$ and if Assumptions~\ref{ass local boundedness}, \ref{ass scheutzow} hold then the solution to~\eqref{eq:examplescheutzow} is unique.
\end{theorem}

We will need the following observation: if $\pi \in \Pi(\mu, \nu)$ 
then, due to the theorem on disintegration, (see for example~\cite[Theorem 5.3.1]{ambrosio2008}) 
there exists a family $(P_{x})_{x\in D} \subset \mathcal P(D)$ such that
\[
\int_{D\times D} f(x,y)\,\pi(dx,dy) = \int_D \left(\int_D f(x,y)\,P_x(dy)\right)\,\mu(dx)
\] 
for any $f=f(x,y)$ which is a $\pi$-integrable function on $D\times D$.
In particular if $f=f(x)$ then 
\[
\int_{D\times D} f(x)\,\pi(dx,dy) = \int_D f(x)\left(\int_D \,P_x(dy)\right)\,\mu(dx) = \int_D f(x)\,\mu(dx)\,.
\]

\begin{proof}
Our aim is to show that Assumption~\ref{ass gmc} holds, since uniqueness then follows from Corollary~\ref{corollary uniqueness}. 
We know, from Lemma~\ref{lemma:uniform} that for any $t\in I$ we have the estimate \newline
$\int_D v(t,x,\mathscr L(x_t)) \, \mathscr L(x_t)(dx) \leq \sup_{t\in I}M(t)$ and so it suffices to verify~\eqref{eq gmc} for measures $\mu$ such that  
$\int_D v(t,x,\mu)\, \mu(dx) \leq \sup_{t\in I}M(t)$.
From Assumption~\ref{ass scheutzow} i), we have
\[
2(x-y)(B(x,\mu)-B(y,\nu))+|\Sigma(x,\mu) - \Sigma(y,\nu)|^2
\leq M(x',y',x'',y'')[|x-y|^2 + |x'-y'|^2 + |x''-y''|^2]\,,
\]
where $x' = \int_D \bar b(z)\mu(dz)$, $y' = \int_D \bar b(z)\nu(dz)$,
$x'' = \int_D \bar \sigma(z)\mu(dz)$ and $y'' = \int_D \bar \sigma(z)\nu(dz)$.
We note that each of $|x'|$,$|y'|$,$|x''|$ and $|y''|$ are in a compact
subset of $\mathbb R$, due to Assumption~\ref{ass scheutzow}  ii), since 
\[
\kappa\left(1+\int_D v(t,z,\mu)\,\mu(dz)\right) +
\kappa\left(1+\int_D v(t,z,\mu)\,\nu(dz)\right) \leq 2\kappa(1+\sup_{t\in I}M(t)) \,.
\]
\secondrefereee{$\kappa$ and $1$ no longer lost on rhs.}
As $M$ maps bounded sets to bounded sets we can choose a constant $g$ 
sufficiently large so that $M(x',y',x'',y'')\leq g$ for all $\mu, \nu$.

We apply the remark on disintegration 
to see that
\[
|x'-y'|^2 = \left|\int_D \bar b(\bar x)\mu(d\bar x) -  \int_D \bar b(\bar y)\nu(d\bar y)\right|^2
= \left|\int_{D\times D} (\bar b(\bar x) - \bar b(\bar y))\,\pi(dx,d\bar y)\right|^2\,.
\]
From Assumption~\ref{ass scheutzow} iii), we get
\[
\begin{split}
|x'-y'|^2 & \leq \kappa^2  \int_{D \times D}  (1+\sqrt{v(t,\bar x,\mu) } + \sqrt{v(t,\bar y,\mu)})^2\, \pi(d\bar x,d\bar y)  \int_{D \times D}  |\bar x-\bar y|^2 \pi(d\bar x,d\bar y) \\
& \leq 3\kappa^2 (1+2\sup_{t\in I}M(t))\int_{D \times D} |\bar x-\bar y|^2 \pi(d\bar x,d\bar y)\,. 
\end{split}
\]
Since the calculation for  $|x''-y''|^2$ is identical we finally obtain
\[
2(x-y)(B(x,\mu)-B(y,\nu))+|\Sigma(x,\mu) - \Sigma(y,\nu)|^2
\leq g|x-y|^2 + 6 g \kappa^2(1+2 \sup_{t\in I}M(t))\int_{D \times D} |\bar x-\bar y|^2 \pi(d\bar x,d\bar y) 
\]
as required to have Assumption~\ref{ass gmc} satisfied with $\bar v(z)=|z|^2$.
\end{proof}
\bothreferees{Constants cleaned up and lyapunov function $\bar v$ explicitly stated. }
 
\section{Invariant Measures}
 \bothreferees{Section rewritten as the first half was not true as the flow property did not hold for the proposed semigroup due to the measure dependence of the coefficients}
We will establish the existence of a stationary measure for semigroups on $C_b(\cP_2(D))$ associated with the flow of laws of solutions to~\eqref{eq mkvsde} where the coefficients $b$ and $\sigma$ do not depend on $t$, via the Krylov--Bogolyubov Theorem (see \cite[Chapter 7]{DaPrato}). One cannot consider a semigroup acting on $C_b(D)$ due to the measure-dependence of the coefficients. Let the conditions of Theorem \ref{thm:weakexistence} hold with suitable assumptions on $m_1$ and $m_2$ such that we are within the regime where $I=[0,\infty )$.
   
Define the semigroup 
$(\mathscr P_t)_{t\geq 0}$ by 
\begin{equation}
\label{eq semigroup for fns of meas}
\mathscr P_t\phi(\mu)=\phi(\mathscr L(x^\mu_t)) \,\,\, 
\text{for $\phi\in C_b(\cP_2(D))$ and $t\geq 0$.} 	
\end{equation}
Here $x^\mu_t$ denotes a solution to~\eqref{eq mkvsde} started from $x_0\sim \mu$.
To ensure that $\mathscr L(x^\mu_t) \in \mathcal P_2(D)$ we assume that 
the conditions of Theorem~\ref{thm:weakexistence} hold with
$V$ satisfying $V(t,x)\geq |x|^2$.
If $D=\mathbb R^d$ then we can apply the chain rule for functions of measures
from e.g.~\cite{buckdahn2017mean} or~\cite{chassagneux2014classical} to obtain
that for $\phi\in  \mathcal C^{(1,1)}(\mathcal P_2(D))$ 
\begin{equation}
\label{eq fpkeq meas}
\begin{split}
& \phi(\mathscr L(x_t)) - \phi(\mathscr L(x_0)) \\
& = \int_0^t\langle \mathscr L(x^{\mu}_s),  b(\cdot,\mathscr L(x^{\mu}_s)) \partial_\mu \phi(\mathscr L(x^{\mu}_s)) + \text{tr}\left[a(\cdot,\mathscr L(x^{\mu}_s)) \partial_y \partial_\mu \phi(\mathscr L(x^{\mu}_s))\right]\rangle\,ds.		
\end{split}
\end{equation}
In the case that $D\subset \mathbb R^d$ we have to assume that there is $k\in \mathbb N$ such that 
$V(t,x) \geq |x|^{2 }$ for $x\in D\setminus D_k$.
We consider first $x^{k,\mu}$ given by~\eqref{eq mkvsde k} started from $\mu^k$ ($\mu$ restricted to $D_k$ with external mass moved to $0$). 
By Proposition~\ref{propn ito for meas only} we have for 
$\phi\in  \mathcal C^{(1,1)}(\mathcal P_2(D))$ that
\begin{equation}
\label{eq fpkeq meas on k}
\begin{split}
& \phi(\mathscr L(x^k_t)) - \phi(\mathscr L(x^k_0))\\
& = \int_0^t\left\langle \mathscr L(x^{k,\mu}_s),  b(\cdot,\mathscr L(x^{k,\mu}_s)) \partial_\mu \phi(\mathscr L(x^{k,\mu}_s)) + \text{tr}\left[a(\cdot,\mathscr L(x^{k,\mu}_s)) \partial_y \partial_\mu \phi(\mathscr L(x^{k,\mu}_s))\right]\right\rangle\,ds.	
\end{split}	
\end{equation}
From Lemma~\ref{lemma:uniform} we get that $\sup_k \sup_t \mathbb E |x^k_t|^{2 } < \infty$. 
Moreover Lemma~\ref{lemma tightness} implies, together
with Prohorov's theorem convergence of a subsequence of the laws 
(and since we know the limit of these is given by~\eqref{eq mkvsde} due
to the proof of Theorem~\ref{thm:weakexistence}). 
We thus have $W_2(\mathscr L(x^k_t), \mathscr L(x_t)) \to 0$
as $k\to \infty$.
Due to continuity of coefficients $b$, $\sigma$ and since $\phi\in  \mathcal C^{(1,1)}(\mathcal P_2(D))$ we can take the limit $k \to \infty$ 
in~\eqref{eq fpkeq meas on k} to obtain~\eqref{eq fpkeq meas}.

The two main conditions for Krylov--Bogolyubov's theorem to hold is that the 
semigroup is Feller and a tightness condition. 
As we are not assuming any non-degeneracy of the diffusion coefficient we cannot
always guarantee that the semigroup is Feller. 
See, however, Lemma~\ref{cor station meas} for a partial result.
\begin{theorem}\label{cor meas station feller}
Let the conditions of Theorem~\ref{thm:weakexistence} hold with $I=[0,\infty)$,
and $V(t,x)\geq|x|^2$ for $x\in D\setminus D_k$ for some $k\in \mathbb N$. 
If the semigroup $(\mathscr P_t)_{t\geq 0}$ given by~\eqref{eq semigroup for fns of meas} is Feller then there exists an invariant measure.
\end{theorem}

We will need the following fact from~\cite{meleard1996asymptotic} to prove this theorem:
Let $S$ be a Polish space and $(m_t)_{t\geq 0} $ be a family of probability measures on $\cP(S)$ i.e. $m_t \in \mathcal P(\mathcal P(S))$. 
Define the intensity measure $I(m_t)$ by
\[
\langle I(m_t),f \rangle=\int_{\cP(S)} \langle \nu,f \rangle \, m_t(d\nu)\,,\,\,\,\,
f \in B(S)\,.
\]
Here $B(S)$ denotes all the bounded measurable functions from $S$ to $\mathbb R$.
Then $(m_t)_{t\geq 0}$ is tight if and only if the family of intensity measures $(I(m_t))_{t\geq 0}\subset \cP(S)$ is tight.

\begin{proof}[Proof of Theorem~\ref{cor meas station feller}]
We recall that $\cP_2(D)$ with the Wasserstein distance $W_2$ is Polish~\cite[Theorem 6.18]{villani2009}. 
Fix $\mu \in \mathcal P_2(D)$ and let $x^\mu$ be a solution to~\eqref{eq mkvsde bar}. 
We note that with $\pi_t(\mu, B) := \delta_{\mathscr L(x^\mu_t)}(B)$ we have, 
from~\eqref{eq semigroup for fns of meas}, that
\[
\mathscr P_t \phi(\mu) = \phi(\mathscr L(x^\mu_t))
= \int_{\mathcal P_2(D)} \delta_{\mathscr L(x^\mu_t)}(\nu)\phi(\nu)\,d\nu
= \int_{\mathcal P_2(D)} \phi(\nu)\,\pi_t(\mu,d\nu)\,.
\]
Define the family of measures $(m^\mu_t)_{t\geq 0} \subset \mathcal P(\mathcal P_2(D))$ 
by 
\[
m^\mu_t(B) := \frac1t \int_0^t \mathbb \pi_s(\mu, B)\,ds = \frac1t \int_0^t \delta_{\mathscr L(x_s^\mu)}(B)\,ds\,,\,\,\,\, B \in \mathscr B(\mathcal P_2(D))\,.
\]
To apply the Krylov--Bogolyubov Theorem we need to show that the family $(m_t)_{t\geq 0}$ is tight. 
We observe that for all $f\in B(D)$ we have
\[
\begin{split}
\int_{\mathcal P(D)} \langle \nu,f \rangle \, m^\mu_t(d\nu) =  
\int_{\cP(D)} \langle \nu,f \rangle \, \frac{1}{t}\int_0^t\delta_{\mathscr L (x^\mu_s)}\,(d \nu)\,ds
&= \frac{1}{t}\int_0^t\langle  \mathscr L (x^\mu_s),f \rangle ds\\
&= \left\langle \frac{1}{t}\int_0^t \mathscr L (x^\mu_s)\,ds, f \right\rangle.  
\end{split}
\]
Therefore $I(m^\mu_t) = \frac{1}{t}\int_0^t \mathscr L (x^\mu_s)\,ds$.
It remains to show that family of intensity measures $(I(m_t))_{t\geq 0}\subset \cP(D)$ is tight. 
For $B \in \mathscr B(D)$ we have 
\[
I(m^\mu_t)(B) = \frac{1}{t}\int_0^t\langle  \mathscr L (x^\mu_s), \mathds 1_B \rangle \,ds = \frac{1}{t}\int_0^t  \mathscr L (x^\mu_s)(B) \, ds
= \frac{1}{t}\int_0^t \mathbb P(x^\mu_s \in B)\, ds\,.
\]
By Fatou's Lemma and Lemma~\ref{lemma tightness} we know that for any $\varepsilon > 0$ there exists sufficiently large $m_0$ such that for all $m>m_0$ we have $\sup_{t\in I}\mathbb P[x^\mu_t\notin D_m]<\varepsilon$. 
Therefore $I(m^\mu_t)(D\setminus D_m)=\frac{1}{t}\int_0^t\mathbb P(x^\mu_s\notin D_m)ds<\varepsilon$ and hence $(I(m^\mu_t))_{t\geq 0}$ is tight.
\end{proof}

\begin{lemma}\label{cor station meas}
Let assumptions of Theorem \ref{thm:weakexistence} hold for $I=[0,\infty)$ along with either Assumption~\ref{ass gmc} 
or~\ref{ass int gmc}.
Assume further that 
\[
W_{\bar v}(\mu,\nu)<\infty \,\,\,\text{ for $\mu,\nu$ in }\,\,\,
\mathcal P_{v}(D):=\bigg\{\mu\in \mathcal P(D) : \int_D v(0,x,\mu)\,\mu(dx)<\infty\bigg\}\,.
\]
Then the semigroup $(\mathscr P_t)_{t\geq 0}$ acting on  $C_b(\cP_v(D) )$ and defined as in~\eqref{eq semigroup for fns of meas} is Feller. 
\end{lemma}
Note that here, we are considering a semigroup acting on space of measures possibly different to that previously considered. In the case where $v$ and $\bar v$ are polynomials, one may replace the assumption of the Feller property in Theorem \ref{cor meas station feller} with the assumptions of Lemma \ref{cor station meas} and $W_{\bar v}(\mu,\nu)< \infty$ for any $\mu,\nu\in \cP_v(D)$ is no longer required.

\begin{proof}
Fix $t\in I$ and $\mu_1, \mu_2 \in \mathcal P_{  v}(D)$.
From the continuous dependence on initial condition, Theorem~\ref{thm contdep}, we have 
\begin{equation}\notag
\begin{alignedat}{1}
W_{\bar v}(\mathscr L(x^{\mu_1}_t),\mathscr L(x^{\mu_2}_t))\leq \mathbb  E [ {\bar v}(x^{\mu_1}_t - x^{\mu_2}_t) ] \leq c_t \mathbb E [ {\bar v}(x^{\mu_1}_0 - x^{\mu_2}_0)] = c_t \int_{D\times D} {\bar v}(x-y) \pi(dx,dy) \,.	
\end{alignedat}
\end{equation}
Taking the infimum over all the possible couplings yields,
\begin{equation}
\label{eq feller for meas 1}
\begin{alignedat}{1}
W_{\bar v}(\mathscr L(x^{\mu_1}_t),\mathscr L(x^{\mu_2}_t))\leq c_t W_{\bar v}(\mu_1,\mu_2).
\end{alignedat}
\end{equation} 
Let $\varepsilon > 0$ be given. 
For any $\phi \in C_b(\cP_v(D) )$ and $\mu\in \cP_v(D)$
there is $\delta_{\phi,\mu}$
such that $W_{\bar v}(\mu,\nu) < \delta_{\phi,\mu}$ implies that
$|\phi(\mu) - \phi(\nu)|< \varepsilon$.
Now, by the uniqueness of solutions $x^{\mu}$, for fixed time $t\in I$, $\delta_{\mu_1}(t) :=c_t^{-1} \delta_{\phi,\mathscr L(x^{\mu_1}_t)} $. 
Then, due to~\eqref{eq feller for meas 1}, if $W_{\bar v}(\mu_1,\mu_2) \leq \delta_{\mu_1}(t)$ then $W_{\bar v}(\mathscr L(x^{\mu_1}_t),\mathscr L(x^{\mu_2}_t))<\delta_{\phi,\mathscr L(x^{\mu_1}_t)}$ and we get $|P_t \phi(\mu_1) - P_t \phi(\mu_2)| < \varepsilon$
as required. 
\end{proof}

\secondrefereee{Indeed, it seemed we were claiming uniform continuity and has been changed.}

\section*{Acknowledgements}
We are grateful to Sandy Davie and  X\={i}l\'{i}ng Zh\={a}ng, both from the University of Edinburgh, for numerous discussions on the topic of this work and many helpful suggestions. Moreover, we are indebted to both referees for careful reading of the manuscript and their comments which helped significantly improve the paper.


\appendix

\section{Measure Derivatives of Lions and Associated It\^o Formula }
\label{sec measure derivatives and ito}

For the construction of the measure derivative in the sense of Lions 
we follow the approach from~\cite[Section 6]{cardaliaguet2012}.
There are three main differences: The first difference is that we define the measure derivative in a domain. More precisely we will define the measure derivative for any measure as long as it has support on $D_k \subset D$ for some 
$k\in \mathbb N$ 
(recall that $\overline{D}_k \subset D_{k+1}$ and $\bigcup_k D_k = D$ and every $D_k$ is bounded and open), in practice for the processes $x^k$, this may be $D_{k+1}$.\color{black}
 
This is precisely what is needed for the analysis in this paper.
The second difference is that we are explicit in making it clear 
why the 
measure derivative is independent of the probability space used to realise the measure
as well as the random variable used. 
The third difference is in proving the ``Structure of the gradient'', 
see~\cite[Theorem 6.5]{cardaliaguet2012}. 
Thanks to an observation by Sandy Davie (University of Edinburgh), we can show 
as part iii) of Proposition~\ref{propn L deriv} that the measure derivative
has the right structure even if it only exists at the point $\mu$ instead of 
for every square integrable measure, as is required in~\cite{cardaliaguet2012}. 
The method of Sandy Davie also conveniently results in a much shorter proof. We assume the same regularity as in \cite{chassagneux2014classical}, but less regularity is assumed than in \cite{buckdahn2017mean} following the observations of \cite{Li2017}
\secondrefereee{Differences of regularity added.}

\subsection{Construction of First-Order Lions' Measure Derivative on $D_k\subset D \subseteq \bR^d$}
Consider $u:\mathcal{P}_2(D) \to \mathbb R$. 
Here $\mathcal{P}_2(D)$ is a space of probability measures on $D$ 
that have second moments i.e. $\int_D x^2 \mu (dx) < \infty$ for
$\mu \in \mathcal{P}_2(D)$.
We want to define the derivative at points $\mu \in \mathcal P_2(D)$ such that $\text{supp}(\mu)\subseteq D_k$.
We shall write $\mu \in \mathcal P(D_k)$ if $\mu$ is a probability measure on $D$ with support in $D_k$. 

\begin{definition}[L-differentiability at $\mu \in \mathcal{P}(D_k)$]
\label{def L differentiability}
We say that $u$ is {\em L-differentiable at} $\mu \in \mathcal{P}(D_k)$ if 
there is an atomless
Polish probability space $(\Omega, \mathcal F, \mathbb P)$ 
and an $X \in L^2(\Omega)$ such that 
$\mu = \mathscr{L}(X)$ and the function $U:L^2(\Omega) \to \mathbb R$ given
by $U(Y) := u(\mathscr{L}(Y))$ is Fr\'echet differentiable at $X$. 
We will call $U$ the {\em lift} of $u$.	
\end{definition}
Clearly 
$\text{supp}(X) \subseteq D_k$ for $\mu \in \mathcal P(D_k)$.
We recall that saying $U:L^2(\Omega; D) \to \mathbb R$ is Fr\'echet differentiable at $X$
with $\text{supp}(X)\subseteq D_k$ means that there exists a bounded linear operator $A:L^2(\Omega) \to \mathbb R$ 
such that for  
\begin{equation}\label{eq frechet}
\lim_{\substack{|Y|_2 \to 0 \\ \text{supp}(X+Y)\subseteq D} } \bigg| \frac{U(X+Y)-U(X)}{|Y|_2} - \frac{AY}{|Y|_2} \bigg| = 0\, .
\end{equation}
Note that 
Since $L^2(\Omega)$ is a Hilbert space with the inner product 
$(X,Y) := \mathbb E [XY]$ we can identify $L^2(\Omega)$ 
with its dual $L^2(\Omega)^*$ via this inner product. 
Then the bounded linear operator $A$ defines an element $DU(X) \in L^2(\Omega)$
through \[(DU(X),Y) := AY \quad \forall Y \in L^2(\Omega).\]

\begin{proposition}
\label{propn L deriv}
Let $u$ be L-differentiable at $\mu \in \mathcal P(D_k)$, 
with some atomless $(\Omega, \mathcal F, \mathbb P)$, lift $U$ and
$X\in L^2(\Omega)$ such that $\mu=\mathscr L(X)$.
Let $(\bar \Omega, \bar{\mathcal F}, \bar {\mathbb P})$ 
be an arbitrary atomless, Polish probability space 
which supports $\bar{X} \in L^2(\bar \Omega)$ and on which 
we have the lift $\bar{U}(Y) := u(\mathscr{L}(Y))$.
Then
\begin{enumerate}[i)]
\item The lift $\bar U$ is Fr\'echet differentiable at $\bar X$ with
derivative $D\bar U(\bar X)\in L^2(\bar \Omega)$.
\item The joint law of $(X,DU(X))$ equals that of $( \bar X ,D\bar U(\bar X))$.
\item There is $\xi:D_k \to D_k$ measurable such that
$\int_{D_k} \xi^2(x) \mu (dx) < \infty$ and almost surely,
\bothreferees{Made explicit that the below equalities are almost sure equalities and not sure.}
\[ \xi(X) = DU(X), \quad \xi(\bar X) = D\bar U(\bar X)\,.\]
\end{enumerate}	
\end{proposition}
Once this is proved we will know that the notion of L-differentiability
depends neither on the probability space used nor on the random variable used.
Moreover the function $\xi$ given by this proposition is again independent 
of the probability space and random variable used. 

\begin{definition}[L-derivative of $u$ at $\mu$]
\label{def L derivative}
If $u$ is L-differentiable at $\mu$ then we write 
$\partial_\mu u(\mu):= \xi$,
where $\xi$ is given by Proposition~\ref{propn L deriv}. 
Moreover we have $\partial_\mu u: \mathcal P_2(D_k) \times D_k \to D_k$ given by
\[
\partial_\mu u(\mu, y) := [\partial_\mu u(\mu)](y)\,.
\]
\end{definition}

To prove Proposition~\ref{propn L deriv} we will need the following result:
\begin{lemma}
\label{lemma tau}
Let $(\Omega, \mathcal F, \mathbb P)$ and $(\bar \Omega, \bar{\mathcal F}, \bar {\mathbb P})$ 
be two atomless, Polish probability spaces supporting $D_k$-valued random variables  $X$ and $\bar X$ such that
$\mathscr{L}(X)	=\mathscr{L}(\bar X)$.
Then for any $\epsilon > 0$ there exists $\tau:\Omega \to \bar \Omega$ 
which is bijective, such that both $\tau$ and $\tau^{-1}$ are measurable
and measure preserving and moreover
\[
|X -  \bar X \circ \tau |_\infty   < \epsilon \,\,\, \text{ and }\,\,\, 
|X \circ \tau^{-1} - \bar X|_\infty   < \epsilon\,.
\]
\end{lemma}
\begin{proof}
Let $(A_n)_n$ be a measurable partition of $D_k$ such that
$\text{diam}(A_n) < \epsilon$.
Let
\[
B_n := \{X \in A_n\},\quad \bar B_n := \{\bar X \in A_n\}\,.
\] 	
These form measurable partitions of $\Omega$ and $\bar \Omega$ respectively 
and moreover $\mathbb P(B_n) = \bar{\mathbb P}(\bar B_n)$.
As the probability spaces are atomless, there exist
$\tau_n :B_n \to \bar B_n$ bijective, such that $\tau_n$ and $\tau_n^{-1}$
are measurable and measure preserving. 
See~\cite[Sec. 41, Theorem C]{hamlos1950} for details
.
Let
\[
\tau(\omega) := \tau_n(\omega)\,\,\, \text{if $\omega \in B_n$},\quad \tau^{-1}(\bar \omega) := \tau_n^{-1}(\bar \omega)\,\,\, \text{if $\bar \omega \in \bar B_n$} \,.
\]
We can see that these are measurable, measure preserving bijections. Now consider
$\omega \in B_n$. Then $\tau(\omega) = \tau_n(\omega) \in \bar B_n$.
But then $X(\omega) \in A_n$ and $\bar{X}(\tau(\omega)) \in A_n$ too.
Hence 
\[
|X(\omega) - \bar{X}(\tau(\omega))| < \epsilon \quad \forall \omega \in \Omega\,.
\]
The estimate for the inverse is proved analogously.
\end{proof}

We use the notation $L^2:=L^2(\Omega)$ and $\bar L^2:= L^2(\bar \Omega)$.

\begin{proof}[Proof of Proposition~\ref{propn L deriv}, part i)]
For any $h>0$ we have $\tau_h, \tau_h^{-1}$ given by Lemma~\ref{lemma tau}
measure preserving and such that $|X -  \bar{X}\circ\tau_h |_\infty < h$.
This means that $|X - \bar{X}\circ\tau_h |_2 < h$ and we have the analogous
estimate with $\tau_h^{-1}$.
Our first aim is to show that $(DU(X)\circ \tau_h^{-1})_{h>0}$ is a
Cauchy sequence in $\bar L^2$.

Fix $\epsilon > 0$. 
Then $\exists \,\delta > 0$ such that we have
\[
|U(X+Y) - U(X) - (DU(X),Y)| < \frac{\epsilon}{2}|Y|_2 \quad \text{for all}\,\,\,|Y|_2 < \delta\,\,\, \text{and}\,\,\,\text{supp}(X+Y)\subseteq D  \,,
\] 	
since $U$ is Fr\'echet differentiable at $X$.
Fix $h, h' < \delta/2$ and consider $|\bar Y|_2 < \delta / 2$ and $\text{supp}(\bar X+\bar Y)\subseteq D$\,.
Then, since the maps $\tau_h^{-1}$ are measure preserving, we have
\[
(DU(X)\circ \tau_h^{-1}, \bar Y) = (DU(X), \bar Y \circ \tau_h)\,.
\]
Note that the inner product on the left is in $\bar L^2$ but the one on the right
is in $L^2$. 
This will not be distinguished in our notation.
Let $Z_h := \bar Y \circ \tau_h - X + \bar X\circ \tau_h$.
Then $|Z_h|_2 \leq |\bar Y|_2 + |\bar X\circ \tau_h - X|_2 <\delta$ and since
$\text{supp}(\bar X + \bar Y) \subseteq D$, we have $\text{supp}(X+ Z_h )\subseteq D$.
Moreover
\begin{equation*}
\begin{split}
 (DU(X)\circ \tau_h^{-1} & - DU(X)\circ \tau_{h'}^{-1},\bar Y)
=  (DU(X), Z_h) - (DU(X),Z_{h'})\\
& + (DU(X), \bar X \circ \tau_h - X) + (DU(X), X - \bar X \circ \tau_{h'} )\\	
= & - U(X+Z_h) + U(X) + (DU(X),Z_h) + [U(X+Z_h) - U(X)]\\
 & + U(X+Z_{h'}) - U(X) - (DU(X),Z_{h'}) - [U(X+Z_{h'}) - U(X)]\\
& + (DU(X), \bar X \circ \tau_h - X) + (DU(X), X - \bar X \circ \tau_{h'} )\,.   
\end{split}
\end{equation*}
But as $\tau_h$ is measure preserving and $U$ and $\bar U$ only depend on the
law, we have 
\[
U(X+Z_h) = U(\bar Y \circ \tau_h + \bar X\circ \tau_h) = \bar U(\bar Y + \bar X) 
= U(X+Z_{h'}).
\]
Hence
\[
\begin{split}
|(DU(X)\circ \tau_h^{-1} & - DU(X)\circ \tau_{h'}^{-1},\bar Y)|
\leq \frac{\epsilon}{2}|Z_{h'}|_2 + \frac{\epsilon}{2}|Z_{h}|_2 + 2|DU(X)|_2\max(h,h') \, \\
& \leq \epsilon |Y|_2 + \epsilon \max(h,h') + 2|DU(X)|_2\max(h,h') \,.	
\end{split}
\]
This means that 
\begin{equation*}
\begin{split}
& |DU(X)\circ \tau_h^{-1} - DU(X)\circ \tau_{h'}^{-1}|_2\\	
& = \sup_{|\bar Y|_2 = \delta/2} \frac{|(DU(X)\circ \tau_h^{-1} - DU(X)\circ \tau_{h'}^{-1},\bar Y)|}{|\bar Y|_2}\leq \epsilon + (2 \epsilon
+ 4|DU(X)|_2) \frac{\max(h,h')}{\delta}\,.
\end{split}
\end{equation*}
Since we can choose $h,h' < \tfrac{\delta}{2}$ and 
also $h,h' < \tfrac{\epsilon \delta}{4|DU(X)|_2}$ we have the required estimate and see
that $(DU(X)\circ \tau_h^{-1})_{h>0}$ is a Cauchy sequence in $\bar L^2$.
Thus, there is $\psi \in \bar L^2$ such that 
\[DU(X)\circ \tau_h^{-1}\to \psi \,\,\, \text{as} \,\,\, h \searrow 0.\]

The next step is to show that $\bar U$ is Fr\'echet differentiable at $\bar X$ 
and $\psi = D\bar U(\bar X)$.
To that end we note that
$\bar U(\bar X + \bar Y) = U(X+Z_h)$ and 
\[
(DU(X),\bar Y\circ \tau_h) = (DU(X),Z_h) + (DU(X),X- \bar X \circ \tau_h ).
\]
Hence
\begin{equation*}
\begin{split}
& |\bar U(\bar X + \bar Y) - \bar U(\bar X) - (\psi,\bar Y)|\\
& = |U(X+Z_h) - U(X) - (DU(X),\bar Y\circ \tau_h) 
+ (DU(X),\bar Y \circ \tau_h) - (\psi,\bar Y)|\\
& \leq \varepsilon|Z_h|_2 + |DU(X)|_2 h + |DU(X)\circ \tau_h^{-1} - \psi||\bar Y|_2
 \leq 4 \epsilon |\bar Y|_2, 
\end{split}	
\end{equation*}
for $h$ sufficiently small. 
Thus $\bar U$ is differentiable at $\bar X$ and  $\psi = D\bar U(\bar X)\in \bar L^2$.
\end{proof}

\begin{proof}[Proof of Proposition~\ref{propn L deriv}, part ii)]
We first note that 
\[
\mathscr{L}(X\circ \tau_h^{-1}, DU(X) \circ \tau_h^{-1}) = \mathscr{L}(X, DU(X))
\]
since the mapping $\tau_h^{-1}$ is measure preserving.
Moreover 
\[
\text{$(X\circ \tau_h^{-1}, DU(X) \circ \tau_h^{-1}) 
\to (\bar X, D \bar U((\bar X)))$ in $L^2(\bar \Omega; \mathbb R^{2d})$ as $h\searrow 0$.}
\]
Hence we get that 
$\mathscr{L}(X,DU(X))= \mathscr{L}(\bar X,D\bar U(\bar X))$.
\end{proof}

\begin{proof}[Proof of Proposition~\ref{propn L deriv}, part iii)]
Note that $\mu$ is not necessarily atomless. Let us first consider the case where it is. 
Then, equipping $(\tilde \Omega,\tilde \cF):=(D_k,\cB(D_k))$ with $\tilde \P:=\mu$, this probability space is atomless and the canonical random element $\tilde X(x):=x$ has law $\mu$. 
Further, there is an $L^2(D_k)$ random element $D\tilde U(\tilde X)$ that is the Fr\'echet derivative of the lift $\tilde U$ at $\tilde X$. Setting $\xi(x)=D\tilde U(\tilde X)(x)$, then by the by uniqueness of the joint distribution, for any probability space $(\bar \Omega, \bar \cF, \bar \P)$ with $\bar X\in L^2(\bar \Omega)$ s.t. $\mathscr L(\bar X)=\mu$, we have  $\xi(\bar X)=D\bar{U}(\bar X)$ almost surely. 

To deal with the case where $\mu$ is not necessarily atomless, we take $\lambda$, the translation invariant measure on $(S^1,\mathscr{B}(S^1))$, with $S^1$ denoting the unit circle. Then, the probability space 
$(\tilde \Omega,\tilde \cF,\tilde \P):=( D_k \times S^1, \mathscr{B}(D_k)\otimes \mathscr{B}(S^1),
\mu \otimes \lambda)$ is atomless.
Let $\tilde L^2$ denote the space of square integrable random variables on this probability space. 
The random variable $\tilde X(x,s):=x$ is in $\tilde L^2$ and has law $\mu$.
With the usual lift $\tilde U$ we know, from part i), 
that $D\tilde U(\tilde X)$ exists in $\tilde L^2$.

For all $t\in S^1$, define the translation operator on $\tilde L^2$, by $R_t\tilde Z(x,s)=\tilde Z(x,s-t)$. Clearly, $R_t\tilde X=\tilde X$. Moreover, $\mathscr L(R_t\tilde Z)=\mathscr L(\tilde Z)$ for all $\tilde Z\in \tilde L^2$ since, by translation invariance of $\lambda$, 
$$ \int_{\tilde \Omega}\1_{B}(R_tZ(x,s))\lambda\otimes\mu(ds,dx)=  \int_{D_k}\int_{S^1}\1_{B}(Z(x,s-t))\lambda(ds)\mu(dx)
=  \int_{D_k}\int_{S^1}\1_{B}(Z(x,s))\lambda(ds)\mu(dx)
.$$ Since the lift $\tilde U(\tilde Z)$ depends only on the distribution of the random element $\tilde Z$, $\tilde U(\tilde Z)=\tilde U(R_t\tilde Z)$ and so $\tilde U(\tilde X+R_t\tilde Z)=\tilde U(\tilde X+\tilde Z)$. Then by uniqueness of the Fr\'echet derivative, recalling equation \eqref{eq frechet}, one can conclude that $\tilde A=\tilde AR_t$ on $\tilde L^2$, for all $t\in S^1$. Therefore, for all $t\in S^1$ and $\tilde Y\in \tilde L^2$, $\tilde AR_t(\tilde Y)=\tilde A(R_t\tilde Y)=\tilde A\tilde Y$ and so,
\begin{equation}\notag
\begin{split}
(D\tilde U(\tilde X),\tilde Y) & =\tilde A\tilde Y   =\tilde AR_t\tilde Y=(D\tilde U(\tilde X),R_t\tilde Y)\\
& =\int_{D_k}\int_{S^1}D\tilde U(\tilde X)(x,s) R_{t} \tilde Y(x,s)\lambda(ds)\mu(dx)\\
& =\int_{D_k}\int_{S^1}R_{-t}D\tilde U(\tilde X)(x,s)  \tilde Y(x,s)\lambda(ds)\mu(dx)=(R_{-t}D\tilde U(\tilde X),\tilde Y).
\end{split}
\end{equation}
Hence, $\tilde \xi(x,s):=D\tilde U(\tilde X)(x,s)$ does not depend on $s$ for $x$ in the support of $\mu$. Write $\xi(x):=\tilde \xi(x,s_0)$ for some $s_0\in S^1$. Then,  
\[
1 = \mu\otimes \lambda \left(\xi(x) = D\tilde U(\tilde X)(x,s)\right) 
= \tilde{\mathbb P}\left(\xi(\tilde X) = D\tilde U (\tilde X)\right)= \bar{\mathbb P}\left(\xi(\bar X) = D\bar U (\bar X)\right)
\]
since $\mathscr{L}(\bar X,D\bar U(\bar X))= \mathscr{L}(\tilde X ,D\tilde U(\tilde X))$ due to part ii) and $\xi(\bar X) = D\bar U(\bar X)$ $\bar{\mathbb P}$-a.s. as required.
\end{proof}
\secondrefereee{Explanation of why $\xi(x,s)=\xi(x)$ added.}

\subsection{Higher-Order Derivatives}
We observe that if $\mu$ is fixed then $\partial_\mu u(\mu)$ is a function
from $D_k\to D_k$.
If, for $y \in D_k$, $\partial_y \left[ \partial_\mu u(\mu)(y)_j \right]$  
exists for each $j=1,\ldots,d$ then 
$\partial_y \partial_\mu u:\mathcal P(D_k) \times D_k \to  D_k \times D_k$ 
is the matrix
\[
\partial_y \partial_\mu u(\mu,y) 
:= \big(\partial_y \left[ \partial_\mu u(\mu)(y)_j \right]\big)_{j=1,\ldots,d} \,\,.
\]

If we fix $y\in D_k$ then $\partial_\mu u(\cdot)(y)$ is a function
from $\mathcal P(D_k) \to D_k$.
Fixing $j=1,\ldots,d$, if
$\partial_\mu u(\cdot)(y)_j : \mathcal P(D_k) \to \mathbb R$ is
L-differentiable at some $\mu$ then its L-derivative is the function
given by part iii) of Proposition~\ref{propn L deriv}, namely
$\partial_\mu \big(\partial_\mu u(\mu)(y)_j \big):D_k \to D_k$.
The second order derivative in measure thus constructed is
$\partial_\mu^2 : \mathcal P(D_k) \times D_k \times D_k \to D_k \times D_k$ given by
\[
\partial_\mu^2(\mu, y, \bar y) := \Big( \partial_\mu \big(\partial_\mu u(\mu,y)_j \big)(\bar y) \Big)_{j=1,\ldots,d}\,\,.
\]

\subsection{It\^o Formula for Functions of Measures}
Assume we have a filtered probability space $(\Omega, \mathcal F, \mathbb P)$ with filtration $(\mathcal F_t)_{t\geq 0}$ satisfying the usual conditions supporting an $(\mathcal F_t)_{t\geq 0}$-Brownian motion $w$ and adapted processes $b$ and $\sigma$ satisfying appropriate integrability conditions. 
We consider the It\^o process  
\[
dx_t = b_t \, dt + \sigma_t dw_t, \,\,\,x_0 \in L^2(\mathcal F_0)
\]
which satisfies $x_t \in D_k$ for all $t$ a.s.

\begin{definition}
\label{def a5}
We say that $u:\mathcal P_2(D) \to \mathbb R$ is in
$\mathcal C^{(1,1)}(\mathcal P_2(D))$ if there is a continuous version of
$y \mapsto \partial_\mu u(\mu)(y)$ such that the mapping 
$\partial_\mu u : \mathcal P_2(D) \times D \to D$ is jointly 
continuous at any $(\mu,y)$ s.t. $y\in \text{supp}(\mu)$ and such that 
$y\mapsto \partial_\mu u(\mu,y)$ is continuously differentiable and its derivative
$\partial_y \partial_\mu u:\mathcal P_2(D) \times D \to  D \times D$ 
is jointly continuous at any $(\mu,y)$ s.t. $y\in \text{supp}(\mu)$.
\end{definition}

The notation $\mathcal C^{(1,1)}$ is chosen to emphasise that we can 
take one measure derivative which is again differentiable (in the usual sense) with respect to the new free variable that arises. 
Note that in~\cite{chassagneux2014classical} such functions are called
partially $\mathcal C^2$.

\begin{proposition}
\label{propn ito for meas only}
Assume that 
\[
\mathbb E \int_0^\infty |b_t|^2 + |\sigma_t|^4 \, dt < \infty\,.
\]	
Let $u$ be in $\mathcal C^{(1,1)}(\mathcal P_2(D))$ 
such that for any compact subset $\mathcal K \subset \mathcal P_2(D)$ 
\begin{equation}
\label{eq ito integrability condition}
\sup_{\mu \in \mathcal K} \int_{D} \left[|\partial_\mu u(\mu)(y)|^2 
+ |\partial_y \partial_\mu u(\mu)(y)|^2 \right]\mu(dy) < \infty \,.	
\end{equation}	
Then, for $\mu_t := \mathscr L(x_t)$,
\[
\begin{split}
u( \mu_t) = & u(\mu_0) 
+ \int_0^t \mathbb E\left[ b_s \partial_\mu u(\mu_s)(x_s) + \frac{1}{2}\text{tr}\left[ \sigma_s \sigma_s^* \partial_y \partial_\mu u(\mu_s)(x_s)\right]\right]\,ds\,.
\end{split}
\]
\end{proposition}

Note that since we are assuming that the process $x$ never leaves some $D_k$, 
we have $\text{supp}(\mu_t) \subseteq D_k$ \color{black}for all times $t$.
The proof relies on replacing $\mu_t$ by an approximation arising as
the empirical measure of $N$ independent copies of the process $x$. 
For marginal empirical measures there is a direct link between measure derivatives 
and partial derivatives, see~\cite[Proposition 3.1]{chassagneux2014classical}.
One can then apply the classical It\^o formula to the approximating system of independent copies of $x$ and take the limit. 
This is done in~\cite[Theorem 3.5]{chassagneux2014classical}.  

Proposition~\ref{propn ito for meas only} can be used to derive an It\^o formula for a function
which depends on $(t,x,\mu)$. 

\begin{definition}
\label{def c122}
By $\mathcal C^{1,2,(1,1)}([0,\infty)\times D \times \mathcal P_2(D))$ 
we denote the functions $v=v(t,x,\mu)$ such 
that $v(\cdot,\cdot,\mu) \in C^{1,2}([0,\infty)\times D)$ for each $\mu$,
and such that $v(t,x,\cdot)$ is in $\mathcal C^{(1,1)}(\mathcal P_2(D))$ for each $(t,x)$.
Moreover all the resulting (partial) derivatives must be jointly continuous 
in $(t,x,\mu)$ or $(t,x,\mu,y)$ as appropriate.	

Finally, by $\mathcal C^{2,(1,1)}(D \times \mathcal P_2(D))$ we denote the
subspace of $\mathcal C^{1,2,(1,1)}([0,\infty)\times D \times \mathcal P_2(D))$ of functions $v$ that are constant in $t$.
\end{definition}

To conveniently express integrals with respect to the laws of the process taken \textit{only}
over the ``new'' variables arising in the measure derivative we introduce
another probability space
$(\tilde \Omega, \tilde{\mathcal F}, \tilde{\mathbb P})$ a
filtration $(\tilde{\mathcal F}_t)_{t\geq 0}$
and processes $\tilde w$, $\tilde b$, $\tilde \sigma$ and a random 
variable $\tilde x_0$  on this probability space
such that they have the same laws as  $w$, $b$, $\sigma$ and $x_0$.
We assume $\tilde w$ is a Wiener process. 
Then 
\[
d\tilde x_t = \tilde b_t \, dt + \tilde \sigma_t d\tilde w_t, \,\,\,\tilde x_0 \in L^2(\tilde{\mathcal F}_0)
\]
is another It\^o process which satisfies $\tilde x_t \in D_k$ for all $t$ a.s.
Moreover, if we now consider the probability space 
$(\Omega\times \tilde \Omega, \mathcal F \otimes \tilde{\mathcal F}, 
\mathbb P \otimes \tilde{\mathbb P})$ 
then we see that the processes with and without tilde are independent on this 
new space.

\begin{proposition}[It\^o formula]
\label{propn ito formula full}
Assume that 
\[
\mathbb E \int_0^\infty |b_t|^2 + |\sigma_t|^4 \, dt < \infty\,.
\]	
Let $v \in \mathcal C^{1,2,(1,1)}([0,\infty)\times D \times \mathcal P_2(D))$ such that for any compact subset $\mathcal K \subset \mathcal P_2(D)$ 
\begin{equation}
\label{eq ito integrability condition tx}
\sup_{t,x,\mu \in [0,\infty)\times D\times\mathcal K} \int_{D} \left[|\partial_\mu v(t,x,\mu)(y)|^2 
+ |\partial_y \partial_\mu v(t,x,\mu)(y)|^2 \right]\mu(dy) < \infty \,.	
\end{equation}
Then, for $\mu_t := \mathscr L(\tilde x_t)$,
\[
\begin{split}
v(t,x_t, \mu_t) -  v(0,x_0,\mu_0) = & \int_0^t \left[\partial_t v(s,x_s,\mu_s) + b_s\partial_x v(s,x_s,\mu_s) 
+ \frac{1}{2}\text{tr}\left[\sigma_s\sigma_s^* \partial_x^2 v(s,x_s,\mu_s)\right] \right]\,ds \\
& + \int_0^t \sigma_s \partial_x v(s,x_s,\mu_s)\, dw_s \\
& + \int_0^t \tilde{\mathbb E}\left[ \tilde b_s \partial_\mu v(s,x_s,\mu_s)(\tilde x_s) + \frac{1}{2}\text{tr}\left[\tilde{\sigma}_s\tilde{\sigma}_s^* \partial_y \partial_\mu v(s,x_s,\mu_s)(\tilde x_s)\right]\right]\,ds\,.
\end{split}
\]
\end{proposition}
Here we follow the argument from~\cite{buckdahn2017mean} explaining how 
to go from an It\^o formula for function of measures only, i.e. from Proposition~\ref{propn ito for meas only}, to the general case. 
Note that it is possible to assume that $\tilde w$, $\tilde b$, $\tilde \sigma$ and $\tilde x_0$ have the same laws as $w$, $b$, $\sigma$
as $x_0$ above, but in fact this is not necessary.
In this paper this generality is needed in the proof of Lemma~\ref{lemma:uniform}.

\begin{proof}[Outline of proof for Proposition~\ref{propn ito formula full}]
Fix $(\bar t,\bar x)$ and apply Proposition~\ref{propn ito for meas only} to the function
$u(\mu):=v(\bar t,\bar x,\mu)$ and the law $\mu_t:=\mathscr L(\tilde x_t)$. 
Then
\[
\begin{split}
v(\bar t,\bar x,\mu_t) - v(\bar t,\bar x,\mu_0) & = 
\int_0^t \tilde{\mathbb E}\left[ \tilde b_s \partial_\mu v(\bar t,\bar x,\mu_s)(\tilde x_s) + \frac{1}{2}\text{tr}\left[\tilde{\sigma}_s\tilde{\sigma}_s^* \partial_y \partial_\mu v(\bar t,\bar x,\mu_s)(\tilde x_s)\right]\right]\,ds \\ 
& =: \int_0^t M(\bar t,\bar x, \mu_s)\,ds\,.
\end{split}
\]
We thus see that the map $t\mapsto v(\bar t, \bar x, \mu_t)$ is absolutely
continuous for all $(\bar t, \bar x)$ and so for almost all $t$ we have $\partial_t v(\bar t, \bar x, \mu_t) = M(\bar t, \bar x, \mu_t)$. 
Note that for completeness we would need to use the definition of $\mathcal C^{1,2,(1,1)}$ functions and a limiting argument to get the partial derivative for all $t$.
See the proof of the corresponding It\^o formula in~\cite{chassagneux2014classical}.
We now consider $\bar v$ given by 
$\bar v(t, x) := v(t,x,\mu_t)$.
Then $\partial_t \bar v(t,x) = (\partial_t v)(t,x,\mu_t) + M(t,x,\mu_t)$.
Using the usual It\^o formula we then have
\[
\begin{split}
\bar v(t,x_t) - \bar v(0,x_0) = & \int_0^t \left[\partial_t v(s,x_s,\mu_s) 
+ M(s,x_s,\mu_s)
+ \frac{1}{2}\text{tr}\left[\sigma_t\sigma_t^* \partial_x^2 v(s,x_s,\mu_s)\right] \right]\,ds \\
& + \int_0^t b_s \partial_x v(s,x_s,\mu_s)\, dw_s\,.
\end{split}
\] 
\end{proof}

\bibliographystyle{imsart-nameyear}
\bibliography{Particles}  

\begin{thebibliography}{37}

\bibitem[\protect\citeauthoryear{Ambrosio, Gigli and
  Savar\'e}{2008}]{ambrosio2008}
\begin{bbook}[author]
\bauthor{\bsnm{Ambrosio},~\bfnm{Luigi}\binits{L.}},
  \bauthor{\bsnm{Gigli},~\bfnm{Nicola}\binits{N.}} \AND
  \bauthor{\bsnm{Savar\'e},~\bfnm{Guiseppe}\binits{G.}}
(\byear{2008}).
\btitle{Gradient Flows in Metric Spaces and in the Space of Probability
  Measures}.
\bpublisher{Birkhauser}.
\end{bbook}
\endbibitem

\bibitem[\protect\citeauthoryear{Billingsley}{1999}]{billingsley}
\begin{bbook}[author]
\bauthor{\bsnm{Billingsley},~\bfnm{Patrick}\binits{P.}}
(\byear{1999}).
\btitle{Convergence of probability measures}.
\bpublisher{John Wiley \& Sons}.
\end{bbook}
\endbibitem

\bibitem[\protect\citeauthoryear{Bogachev, R{\"o}ckner and
  Shaposhnikov}{2016}]{bogachev2016distances}
\begin{barticle}[author]
\bauthor{\bsnm{Bogachev},~\bfnm{V.~I}\binits{V.~I.}},
  \bauthor{\bsnm{R{\"o}ckner},~\bfnm{M}\binits{M.}} \AND
  \bauthor{\bsnm{Shaposhnikov},~\bfnm{S.~V}\binits{S.~V.}}
(\byear{2016}).
\btitle{Distances between transition probabilities of diffusions and
  applications to nonlinear {F}okker--{P}lanck--{K}olmogorov equations}.
\bjournal{Journal of Functional Analysis}
\bvolume{271}
\bpages{1262--1300}.
\end{barticle}
\endbibitem

\bibitem[\protect\citeauthoryear{Bogachev, Röckner and
  Shaposhnikov}{2019}]{bogachevRocknerShaposhnikovStationary2019}
\begin{barticle}[author]
\bauthor{\bsnm{Bogachev},~\bfnm{Vladimir~I.}\binits{V.~I.}},
  \bauthor{\bsnm{Röckner},~\bfnm{Michael}\binits{M.}} \AND
  \bauthor{\bsnm{Shaposhnikov},~\bfnm{Stanislav~V.}\binits{S.~V.}}
(\byear{2019}).
\btitle{Convergence in variation of solutions of nonlinear
  Fokker–Planck–Kolmogorov equations to stationary measures}.
\bjournal{Journal of Functional Analysis}
\bvolume{276}
\bpages{3681 - 3713}.
\bdoi{https://doi.org/10.1016/j.jfa.2019.03.014}
\end{barticle}
\endbibitem

\bibitem[\protect\citeauthoryear{Bogachev
  et~al.}{2015}]{BogachevKrylovRocknerShaposhnikov}
\begin{bbook}[author]
\bauthor{\bsnm{Bogachev},~\bfnm{V.~I.}\binits{V.~I.}},
  \bauthor{\bsnm{Krylov},~\bfnm{N.~V.}\binits{N.~V.}},
  \bauthor{\bsnm{R\"ockner},~\bfnm{M.}\binits{M.}} \AND
  \bauthor{\bsnm{Shaposhnikov},~\bfnm{S.~V.}\binits{S.~V.}}
(\byear{2015}).
\btitle{{F}okker--{P}lanck--{K}olmogorov Equations}.
\bpublisher{AMS}.
\end{bbook}
\endbibitem

\bibitem[\protect\citeauthoryear{Bolley, Ca\~nizo and
  Carrillo}{2011}]{MR2860672}
\begin{barticle}[author]
\bauthor{\bsnm{Bolley},~\bfnm{Fran\c{c}ois}\binits{F.}},
  \bauthor{\bsnm{Ca\~nizo},~\bfnm{Jos{\'e}~A.}\binits{J.~A.}} \AND
  \bauthor{\bsnm{Carrillo},~\bfnm{Jos{\'e}~A.}\binits{J.~A.}}
(\byear{2011}).
\btitle{Stochastic mean-field limit: non-{L}ipschitz forces and swarming}.
\bjournal{Math. Models Methods Appl. Sci.}
\bvolume{21}
\bpages{2179--2210}.
\bdoi{10.1142/S0218202511005702}
\bmrnumber{2860672}
\end{barticle}
\endbibitem

\bibitem[\protect\citeauthoryear{Bolley, Guillin and
  Villani}{2007}]{Bolley2007}
\begin{barticle}[author]
\bauthor{\bsnm{Bolley},~\bfnm{Fran\c{c}ois}\binits{F.}},
  \bauthor{\bsnm{Guillin},~\bfnm{Arnaud}\binits{A.}} \AND
  \bauthor{\bsnm{Villani},~\bfnm{C{\'e}dric}\binits{C.}}
(\byear{2007}).
\btitle{Quantitative concentration inequalities for empirical measures on
  non-compact spaces}.
\bjournal{Probability Theory and Related Fields}
\bvolume{137}
\bpages{541--593}.
\bmrnumber{2280433}
\end{barticle}
\endbibitem

\bibitem[\protect\citeauthoryear{Bossy, Jabir and
  Talay}{2011}]{bossy2011conditional}
\begin{barticle}[author]
\bauthor{\bsnm{Bossy},~\bfnm{Mireille}\binits{M.}},
  \bauthor{\bsnm{Jabir},~\bfnm{Jean-Fran{\c{c}}ois}\binits{J.-F.}} \AND
  \bauthor{\bsnm{Talay},~\bfnm{Denis}\binits{D.}}
(\byear{2011}).
\btitle{{On conditional McKean Lagrangian stochastic models}}.
\bjournal{Probability theory and related fields}
\bvolume{151}
\bpages{319--351}.
\end{barticle}
\endbibitem

\bibitem[\protect\citeauthoryear{Buckdahn et~al.}{2017}]{buckdahn2017mean}
\begin{barticle}[author]
\bauthor{\bsnm{Buckdahn},~\bfnm{Rainer}\binits{R.}},
  \bauthor{\bsnm{Li},~\bfnm{Juan}\binits{J.}},
  \bauthor{\bsnm{Peng},~\bfnm{Shige}\binits{S.}} \AND
  \bauthor{\bsnm{Rainer},~\bfnm{Catherine}\binits{C.}}
(\byear{2017}).
\btitle{Mean-field stochastic differential equations and associated {PDE}s}.
\bjournal{The Annals of Probability}
\bvolume{45}
\bpages{824--878}.
\end{barticle}
\endbibitem

\bibitem[\protect\citeauthoryear{Cardaliaguet}{2013}]{cardaliaguet2012}
\begin{bbooklet}[author]
\bauthor{\bsnm{Cardaliaguet},~\bfnm{Pierre}\binits{P.}}
(\byear{2013}).
\btitle{Notes on Mean Field Games (from {P}.-{L}. {L}ions' lectures at
  {C}oll\`ege de {F}rance)}.
\bhowpublished{Online at
  \url{https://www.ceremade.dauphine.fr/~cardaliaguet/MFG20130420.pdf}}.
\end{bbooklet}
\endbibitem

\bibitem[\protect\citeauthoryear{Carmona and
  Delarue}{2017}]{carmona2017probabilistic}
\begin{bbook}[author]
\bauthor{\bsnm{Carmona},~\bfnm{Rene}\binits{R.}} \AND
  \bauthor{\bsnm{Delarue},~\bfnm{Fran{\c{c}}ois}\binits{F.}}
(\byear{2017}).
\btitle{Probabilistic Theory of Mean Field Games with Applications I-II}.
\bpublisher{Springer}.
\end{bbook}
\endbibitem

\bibitem[\protect\citeauthoryear{Chassagneux, Crisan and
  Delarue}{2014}]{chassagneux2014classical}
\begin{barticle}[author]
\bauthor{\bsnm{Chassagneux},~\bfnm{Jean-Fran{\c{c}}ois}\binits{J.-F.}},
  \bauthor{\bsnm{Crisan},~\bfnm{Dan}\binits{D.}} \AND
  \bauthor{\bsnm{Delarue},~\bfnm{Fran{\c{c}}ois}\binits{F.}}
(\byear{2014}).
\btitle{Classical solutions to the master equation for large population
  equilibria}.
\bjournal{arXiv:1411.3009}.
\end{barticle}
\endbibitem

\bibitem[\protect\citeauthoryear{Da~Prato}{2006}]{DaPrato}
\begin{bbook}[author]
\bauthor{\bsnm{Da~Prato},~\bfnm{Giuseppe}\binits{G.}}
(\byear{2006}).
\btitle{An Introduction to Infinite-Dimensional Analysis}.
\bseries{Universitext}.
\bpublisher{Springer-Verlag Berlin Heidelberg}.
\end{bbook}
\endbibitem

\bibitem[\protect\citeauthoryear{de~Raynal}{2020}]{de2015strong}
\begin{barticle}[author]
\bauthor{\bparticle{de}
  \bsnm{Raynal},~\bfnm{Paul-Eric~Chaudru}\binits{P.-E.~C.}}
(\byear{2020}).
\btitle{Strong well-posedness of {M}c{K}ean-{V}lasov stochastic differential
  equation with H{\"o}lder drift}.
\bjournal{Stochastic Processes and Their Applications}
\bvolume{130}
\bpages{79--107}.
\end{barticle}
\endbibitem

\bibitem[\protect\citeauthoryear{dos Reis, Salkeld and
  Tugaut}{2019}]{DosReis2017LDP}
\begin{barticle}[author]
\bauthor{\bparticle{dos} \bsnm{Reis},~\bfnm{G.}\binits{G.}},
  \bauthor{\bsnm{Salkeld},~\bfnm{W.}\binits{W.}} \AND
  \bauthor{\bsnm{Tugaut},~\bfnm{J.}\binits{J.}}
(\byear{2019}).
\btitle{Freidlin--{W}entzell {LDP}s in path space for {M}c{K}ean--{V}lasov
  equations and the functional iterated logarithm law}.
\bjournal{Annals of Applied Probability}
\bvolume{29}
\bpages{1487--1540}.
\end{barticle}
\endbibitem

\bibitem[\protect\citeauthoryear{Dudley}{2002}]{dudleybook}
\begin{bbook}[author]
\bauthor{\bsnm{Dudley},~\bfnm{R.~M.}\binits{R.~M.}}
(\byear{2002}).
\btitle{Real Analysis and Probability}.
\bpublisher{Cambridge Univ Press}.
\end{bbook}
\endbibitem

\bibitem[\protect\citeauthoryear{Fournier and
  Jourdain}{2017}]{FournierJourdain17}
\begin{barticle}[author]
\bauthor{\bsnm{Fournier},~\bfnm{Nicolas}\binits{N.}} \AND
  \bauthor{\bsnm{Jourdain},~\bfnm{Benjamin}\binits{B.}}
(\byear{2017}).
\btitle{Stochastic particle approximation of the {K}eller--{S}egel equation and
  two-dimensional generalization of {B}essel processes}.
\bjournal{Ann. Appl. Probab.}
\bvolume{27}
\bpages{2807--2861}.
\bmrnumber{3719947}
\end{barticle}
\endbibitem

\bibitem[\protect\citeauthoryear{Funaki}{1984}]{funaki1984certain}
\begin{barticle}[author]
\bauthor{\bsnm{Funaki},~\bfnm{Tadahisa}\binits{T.}}
(\byear{1984}).
\btitle{A certain class of diffusion processes associated with nonlinear
  parabolic equations}.
\bjournal{Probability Theory and Related Fields}
\bvolume{67}
\bpages{331--348}.
\end{barticle}
\endbibitem

\bibitem[\protect\citeauthoryear{Gy{\"o}ngy and
  Krylov}{1996}]{gyongy1996existence}
\begin{barticle}[author]
\bauthor{\bsnm{Gy{\"o}ngy},~\bfnm{Istv{\'a}n}\binits{I.}} \AND
  \bauthor{\bsnm{Krylov},~\bfnm{Nicolai}\binits{N.}}
(\byear{1996}).
\btitle{Existence of strong solutions for {I}t{\^o}'s stochastic equations via
  approximations}.
\bjournal{Probability {T}heory and {R}elated {F}ields}
\bvolume{105}
\bpages{143--158}.
\end{barticle}
\endbibitem

\bibitem[\protect\citeauthoryear{Gärtner}{1988}]{Gartner}
\begin{barticle}[author]
\bauthor{\bsnm{Gärtner},~\bfnm{Jürgen}\binits{J.}}
(\byear{1988}).
\btitle{On the McKean-Vlasov Limit for Interacting Diffusions}.
\bjournal{Mathematische Nachrichten}
\bvolume{137}
\bpages{197-248}.
\end{barticle}
\endbibitem

\bibitem[\protect\citeauthoryear{Hamlos}{1950}]{hamlos1950}
\begin{bbook}[author]
\bauthor{\bsnm{Hamlos},~\bfnm{Paul~Richard}\binits{P.~R.}}
(\byear{1950}).
\btitle{Measure Theory}.
\bpublisher{Van Nostrand}.
\end{bbook}
\endbibitem

\bibitem[\protect\citeauthoryear{Kac}{1956}]{Kac56}
\begin{binproceedings}[author]
\bauthor{\bsnm{Kac},~\bfnm{M.}\binits{M.}}
(\byear{1956}).
\btitle{Foundations of kinetic theory}.
In \bbooktitle{Proceedings of the {T}hird {B}erkeley {S}ymposium on
  {M}athematical {S}tatistics and {P}robability, 1954--1955, vol. {III}}
\bpages{171--197}.
\bpublisher{University of California Press, Berkeley and Los Angeles}.
\bmrnumber{0084985}
\end{binproceedings}
\endbibitem

\bibitem[\protect\citeauthoryear{Khasminskii}{1980}]{Khasminskii80}
\begin{bbook}[author]
\bauthor{\bsnm{Khasminskii},~\bfnm{Rafail}\binits{R.}}
(\byear{1980}).
\btitle{Stochastic stability of differential equations}.
\bseries{Monographs and Textbooks on Mechanics of Solids and Fluids: Mechanics
  and Analysis}
\bvolume{7}.
\bpublisher{Sijthoff \&\ Noordhoff, Alphen aan den Rijn---Germantown, Md.}
\bnote{Translated from the Russian by D. Louvish}.
\bmrnumber{600653}
\end{bbook}
\endbibitem

\bibitem[\protect\citeauthoryear{Li and Min}{2017a}]{MR3609379}
\begin{barticle}[author]
\bauthor{\bsnm{Li},~\bfnm{Juan}\binits{J.}} \AND
  \bauthor{\bsnm{Min},~\bfnm{Hui}\binits{H.}}
(\byear{2017}a).
\btitle{Weak solutions of mean-field stochastic differential equations}.
\bjournal{Stoch. Anal. Appl.}
\bvolume{35}
\bpages{542--568}.
\bdoi{10.1080/07362994.2017.1278706}
\bmrnumber{3609379}
\end{barticle}
\endbibitem

\bibitem[\protect\citeauthoryear{Li and Min}{2017b}]{Li2017}
\begin{barticle}[author]
\bauthor{\bsnm{Li},~\bfnm{Juan}\binits{J.}} \AND
  \bauthor{\bsnm{Min},~\bfnm{Hui}\binits{H.}}
(\byear{2017}b).
\btitle{Weak solutions of mean-field stochastic differential equations}.
\bjournal{Stochastic Analysis and Applications}
\bvolume{35}
\bpages{542-568}.
\bdoi{10.1080/07362994.2017.1278706}
\end{barticle}
\endbibitem

\bibitem[\protect\citeauthoryear{Lu\c{c}on and Stannat}{2014}]{MR3226169}
\begin{barticle}[author]
\bauthor{\bsnm{Lu\c{c}on},~\bfnm{Eric}\binits{E.}} \AND
  \bauthor{\bsnm{Stannat},~\bfnm{Wilhelm}\binits{W.}}
(\byear{2014}).
\btitle{Mean field limit for disordered diffusions with singular interactions}.
\bjournal{Ann. Appl. Probab.}
\bvolume{24}
\bpages{1946--1993}.
\bdoi{10.1214/13-AAP968}
\bmrnumber{3226169}
\end{barticle}
\endbibitem

\bibitem[\protect\citeauthoryear{McKean}{1966}]{McKean66}
\begin{barticle}[author]
\bauthor{\bsnm{McKean},~\bfnm{H.~P.}\binits{H.~P.} \bsuffix{Jr.}}
(\byear{1966}).
\btitle{A class of {M}arkov processes associated with nonlinear parabolic
  equations}.
\bjournal{Proc. Nat. Acad. Sci. U.S.A.}
\bvolume{56}
\bpages{1907--1911}.
\bmrnumber{0221595}
\end{barticle}
\endbibitem

\bibitem[\protect\citeauthoryear{M\'{e}l\'{e}ard}{1996}]{meleard1996asymptotic}
\begin{bincollection}[author]
\bauthor{\bsnm{M\'{e}l\'{e}ard},~\bfnm{Sylvie}\binits{S.}}
(\byear{1996}).
\btitle{{Asymptotic behaviour of some interacting particle systems;
  McKean-Vlasov and Boltzmann models}}.
In \bbooktitle{Probabilistic models for nonlinear partial differential
  equations}
\bpages{42--95}.
\bpublisher{Springer}.
\end{bincollection}
\endbibitem

\bibitem[\protect\citeauthoryear{Mishura and
  Veretennikov}{2016}]{mishura2016existence}
\begin{barticle}[author]
\bauthor{\bsnm{Mishura},~\bfnm{Yuliya~S}\binits{Y.~S.}} \AND
  \bauthor{\bsnm{Veretennikov},~\bfnm{Alexander~Yu}\binits{A.~Y.}}
(\byear{2016}).
\btitle{Existence and uniqueness theorems for solutions of {M}c{K}ean--{V}lasov
  stochastic equations}.
\bjournal{arXiv:1603.02212}.
\end{barticle}
\endbibitem

\bibitem[\protect\citeauthoryear{Scheutzow}{1987}]{Scheutzow87}
\begin{barticle}[author]
\bauthor{\bsnm{Scheutzow},~\bfnm{Michael}\binits{M.}}
(\byear{1987}).
\btitle{Uniqueness and nonuniqueness of solutions of {V}lasov--{M}c{K}ean
  equations}.
\bjournal{J. Austral. Math. Soc. Ser. A}
\bvolume{43}
\bpages{246--256}.
\bmrnumber{896631}
\end{barticle}
\endbibitem

\bibitem[\protect\citeauthoryear{Skorokhod}{1965}]{skorokhod1965}
\begin{bbook}[author]
\bauthor{\bsnm{Skorokhod},~\bfnm{A.~V.}\binits{A.~V.}}
(\byear{1965}).
\btitle{Studies in the theory of random processes}.
\bpublisher{Addison--Wesley, Reading, Mass.}
\end{bbook}
\endbibitem

\bibitem[\protect\citeauthoryear{Stroock and
  Varadhan}{2006}]{StroockVaradhan2006}
\begin{bbook}[author]
\bauthor{\bsnm{Stroock},~\bfnm{Daniel~W.}\binits{D.~W.}} \AND
  \bauthor{\bsnm{Varadhan},~\bfnm{S.~R.~Srinivasa}\binits{S.~R.~S.}}
(\byear{2006}).
\btitle{Multidimensional diffusion processes}.
\bseries{Classics in Mathematics}.
\bpublisher{Springer-Verlag, Berlin}
\bnote{Reprint of the 1997 edition}.
\bmrnumber{2190038}
\end{bbook}
\endbibitem

\bibitem[\protect\citeauthoryear{Sznitman}{1991}]{sznitman1991topics}
\begin{bbook}[author]
\bauthor{\bsnm{Sznitman},~\bfnm{Alain-Sol}\binits{A.-S.}}
(\byear{1991}).
\btitle{Topics in propagation of chaos}.
\bpublisher{Springer}.
\end{bbook}
\endbibitem

\bibitem[\protect\citeauthoryear{Szpruch and Zh\=ang}{2018}]{SzpruchZhang18}
\begin{barticle}[author]
\bauthor{\bsnm{Szpruch},~\bfnm{{\L}ukasz}\binits{{\L}.}} \AND
  \bauthor{\bsnm{Zh\=ang},~\bfnm{X{\i}l{\'\i}ng}\binits{X.}}
(\byear{2018}).
\btitle{{$V$}-integrability, asymptotic stability and comparison property of
  explicit numerical schemes for non-linear {SDE}s}.
\bjournal{Math. Comp.}
\bvolume{87}
\bpages{755--783}.
\bmrnumber{3739216}
\end{barticle}
\endbibitem

\bibitem[\protect\citeauthoryear{Vallender}{1972}]{vallender1972calculation}
\begin{barticle}[author]
\bauthor{\bsnm{Vallender},~\bfnm{S.~S.}\binits{S.~S.}}
(\byear{1972}).
\btitle{Calculation of the Wasserstein Distance Between Probability
  Distributions on the Line}.
\bjournal{Theory Probab. Appl.}
\bvolume{18}.
\end{barticle}
\endbibitem

\bibitem[\protect\citeauthoryear{Villani}{2009}]{villani2009}
\begin{bbook}[author]
\bauthor{\bsnm{Villani},~\bfnm{C\'edric}\binits{C.}}
(\byear{2009}).
\btitle{Optimal transport old and new}.
\bpublisher{Springer}.
\end{bbook}
\endbibitem

\bibitem[\protect\citeauthoryear{Wang}{2018}]{wangDDSDELandau2018}
\begin{barticle}[author]
\bauthor{\bsnm{Wang},~\bfnm{Feng-Yu}\binits{F.-Y.}}
(\byear{2018}).
\btitle{Distribution dependent SDEs for Landau type equations}.
\bjournal{Stochastic Processes and their Applications}
\bvolume{128}
\bpages{595 - 621}.
\bdoi{https://doi.org/10.1016/j.spa.2017.05.006}
\end{barticle}
\endbibitem

\end{thebibliography}
\end{document}